\documentclass[a4paper,12pt]{article}
\usepackage{amsmath}
\usepackage{amssymb}
\usepackage{amsthm}
\usepackage{mathrsfs}

\newtheorem{theorem}{Theorem}[section]
\newtheorem{lemma}[theorem]{Lemma}
\newtheorem{proposition}[theorem]{Proposition}
\newtheorem{definition}[theorem]{Definition}

\newcommand{\proba}{{\mathbb{P}}}

\font\gfont=cmmi10 scaled \magstep2

\newcommand{\traits}{\mathbb{T}}
\newcommand{\atomiques}{\mathcal{A}}
\newcommand{\entiers}{\mathbb{N}}

\newcommand{\borel}{\mathscr{M}}

\newcommand{\petit}{\mathscr{P}}

\newcommand{\EE}{\mathbb{E}}
\newcommand{\ZZ}{\mathbb{Z}}
\newcommand{\NN}{\mathbb{N}}
\newcommand{\PP}{\mathbb{P}}

\newcommand{\RR}{\mathbb{R}}

\newcommand{\B}{{\cal B}}

\newcommand{\N}{{\cal N}}
\newcommand{\Si}{\mathscr{P}}

\newcommand{\X}{{\cal X}}

\newcommand{\gun}{\hbox{\gfont\char49}}

\newcommand{\espace}{\mathcal{A}}

\newcommand{\etat}{\vec\eta}

\newcommand{\hypo}{\mathscr{H}}

\newcommand{\real}{\mathbb{R}}

\newcommand{\sm}{{s-}}

\newcommand{\intot}{\displaystyle \int _0^t }

\newcommand{\indiq}{{\bf 1}}

\newcommand{\rit}{\mathbb{R}}

\newcommand{\tm}{\lambda_{*}}
\newcommand{\tn}{B^{*}}

\newcommand{\convex}{\mathscr{K}}

\newcommand{\tnu}{\tilde\nu}

\title {Quasi-stationary distributions for structured birth and
 death processes with mutations}
\author{Pierre Collet
\and Servet Mart{\'\i}nez \and Sylvie M\'el\'eard \and Jaime San
Mart{\'\i}n}

\begin{document}

\maketitle

\begin{abstract}
We study the probabilistic evolution of a  birth and death
continuous time measure-valued process with mutations and
ecological interactions. The individuals are characterized by
(phenotypic) traits that take values in a compact metric space.
 Each  individual can die or generate a new individual.
The birth and death rates may depend on the environment through
the action of the whole population. The offspring can have the
same trait or can mutate to a   randomly distributed trait. We
assume that the population will be extinct almost surely. Our goal
is the study, in this infinite dimensional framework, of
quasi-stationary distributions when the process is conditioned on
non-extinction. We firstly show in this general setting, the
existence of quasi-stationary distributions. This result is based
on an abstract theorem proving the existence of finite
eigenmeasures for some positive operators. We then consider a
population with constant birth and death rates per individual and
prove that there exists a unique quasi-stationary distribution
with maximal exponential decay rate. The proof of uniqueness is
based on  an absolute continuity property with respect to a
reference measure.
\end{abstract}

\bigskip

\textit{ MSC 2000 subject.} Primary 92D25; secondary  60K35,
60J70,  60J80.\\

\section{ Introduction and main results }
\label{sec1}

\subsection{Introduction}
\label{subsub11} We consider a general discrete model describing a
structured population with a  microscopic individual-based and
stochastic point of view. The dynamics takes into account all
reproduction and death events. Each individual is characterized by
an heritable quantitative  parameter, usually called {\it trait},
which can for example be the expression of its genotype or
phenotype. During the reproduction process, mutations of the trait
can occur, implying some variability in the trait space. Moreover,
the individuals can die. In the general model, the individual
reproduction and death rates, as well  as the mutation
distribution, depend on the trait of the individual  and on the
whole population. In particular, cooperation or competition
between individuals in this population are taken into account.

\medskip

In our model the set of traits ${\traits}$ is a compact metric
space with metric $d$. For convenience we assume
$\hbox{diameter}(\traits)=1$. Let ${\cal B}(\traits)$ be the class
of Borel sets in $\traits$. The structured population is described
by a finite point measure on ${\traits}$.  Thus, the state space,
denoted by ${\atomiques}$, is the set of all finite point measures which is contained in
${\cal M}({\traits})$, the set of positive measures on ${\traits}$.

\medskip

A configuration $\eta\in {\atomiques}$ is described by $(\eta_y:
y\in \traits)$ with $\eta_y\in \ZZ_+=\{0,1,...\}$, where only a
finite subset of elements $y\in \traits$ satisfy $\eta_y>0$. The
finite set of present  traits (i.e. traits of alive individuals)
is denoted by
$$
\{\eta \}:=\{y\in {\traits}: \eta_y >0\}
$$
and called the support of $\eta$. For a function $f$ defined on the trait
space $\traits$, we will denote the integral of $f$ with respect
to $\eta$ by
$$\langle \eta, f \rangle = \sum_{y\in \{\eta \}} f(y) \eta_y.
$$
Let $|\cdot|$ be the cardinal number of a set. We denote by
$\#{\eta}=|\{\eta\}|$ the number of active traits and by
$\|\eta\|=\sum\limits_{y\in \{\eta\}}\eta_y$ the total number of
individuals in $\eta$. The void configuration is denoted
by $\eta=0$, so $\#0=\|0\|=0$ and  we define
${\atomiques}^{-0}:={\atomiques} \setminus\{0\}$
the set of nonempty configurations.
\medskip

The structured population dynamics is given by an individual-based
model, taking into account each (clonal or mutation) birth and
death events.

\medskip

The {\it clonal birth rate}, the {\it mutation birth rate} and the
{\it death rate} of an individual with trait $y$ and a population
$\eta\in{\atomiques}$, are denoted respectively by $b_{y}(\eta)$,
$m_y(\eta)$ and $\lambda_{y}(\eta)$.
The total reproduction rate
for an individual with trait $y\in \{\eta\}$ is equal to
$b_{y}(\eta) + m_{y}(\eta)$.
We assume $\lambda_y(0)=b_y(0)=m_y(0)=0$ for all $y\in \traits$,
which is natural for population dynamics.
In what follows we assume that the functions
\begin{eqnarray}
\label{cont111} \lambda_{y}(\eta), \, b_{y}(\eta), \,
m_{y}(\eta)\,: \traits \times \atomiques^{-0} \to \mathbb{R}_+ \,
\hbox{ are continuous and strictly positive.}
\end{eqnarray}

Let  $\sigma$ be a fixed {\it non-atomic probability measure} on
$({\traits}, {\cal B}({\traits}))$. The {\it density location}
function of the mutations is
 $g:{\traits}\times {\traits}:\to \RR_+$, $(y,z)\to
g_y(z)$, where $g_y(\cdot)$ is the probability density of the
trait of the new mutated individual born from $y$. It satisfies
\begin{equation}
\int\limits_{\traits} g_y(z) d\sigma(z)=1\, \hbox{ for all } y\in
\traits \,.
\end{equation}
We assume that the function $g_\cdot(\cdot)$ is jointly
continuous. To simplify notations we express the mutation part
using location kernel $G(\eta,z):{\atomiques} \times {\cal
B}({\traits}) \to \RR^+$ given by
\begin{equation}
\label{hypac} \forall \,\eta\in \atomiques \, \forall z\in \traits
\;,\; G(\eta, z)= \sum\limits_{y\in \{\eta\}} \eta_y
m_y(\eta)g_y(z) = \langle \eta, m_\cdot(\eta) g_\cdot(z)\rangle.
\end{equation}
 Note that the ratio
$G(\eta,z) d\sigma(z)/\int\limits_{\traits}G(\eta, z)d\sigma(z)$
is the probability that, given  there is a mutation from $\eta$,
the new trait is located at $z$.
 Hypothesis (\ref{cont111})  and  Lemma \ref{general-cont} stated
below imply that the function $G$ is continuous on $\atomiques
\times \traits$.

\bigskip

We  define a continuous time pure jump Markov process $Y=(Y_t)$
taking values on $\atomiques$. We denote by
$Q:\atomiques\times \B(\atomiques)\to \RR_+$, $(\eta, B)\to
Q(\eta, B)$, the kernel of measure jump rates given by
\begin{equation}
\label{rates11}
Q(\eta, B)= \sum\limits_{y\in \{\eta\}, \eta+\delta_y\in B}\!
\eta_y b_{y}(\eta) + \sum\limits_{y\in \{\eta\}, \eta-\delta_y\in
B} \eta_y \lambda_{y}(\eta)+\!\! \int\limits_{\eta+\delta_z\in
B}\!\!\!\! G(\eta,z) d\sigma(z) \,.
\end{equation}
The total
mass $Q(\eta)$ of the kernel at $\eta$ is always finite and given
by

\begin{equation}
\label{masaI} Q(\eta)= \sum\limits_{y\in \{\eta\}}\! \left(
Q({\eta, \eta\!+\!\delta_y}) +Q({\eta,
\eta\!-\!\delta_y})\right)+\!\! \int\limits_{\traits \setminus
\{\eta\}}\!\!\!\! Q({\eta, \eta\!+\!\delta_z}) d\sigma(z) \,.
\end{equation}
  Observe that because of \eqref{cont111}
\begin{eqnarray}
\label{locbo1} \forall k \ge 1\ ,\, Q_+(k)=\sup\{Q(\eta): \|\eta\|
\le k\} < \infty\,.
\end{eqnarray}
The construction of a process $Y$ with c\`adl\`ag trajectories
associated with the kernel $Q$, is the canonical one. Assume that
the process starts from $Y_0=\eta$. Then, after an exponential
time of parameter $Q({\eta})$, the process jumps to
$\eta+\delta_y$ for $y\in \{\eta \}$ with probability $Q({\eta,
\eta+\delta_y})/Q({\eta})$, or to $\eta-\delta_y$ for $y\in \{\eta
\}$ with probability $Q({\eta, \eta-\delta_y})/Q({\eta})$, or to a
point $\eta+\delta_z$ for $z\in \traits\setminus \{\eta\}$ with
probability density $Q({\eta, \eta+\delta_z})/Q({\eta})$ with
respect to $\sigma$. The process restarts independently at the new
configuration.

The process $Y$ can have explosions. To avoid this phenomenon and other
reasons, throughout the paper we shall assume that
\begin{eqnarray}
\label{co333} &&\tn=\sup_{\eta\in \atomiques}\sup_{y\in \{\eta\}}
\big(b_{y}(\eta)+m_{y}(\eta)\big) <\infty.
\end{eqnarray}
This condition also guarantees the existence of the process $(Y_t: t\ge 0)$
as the unique solution of a stochastic
differential equation driven by Poisson point measures. This is done
in Section \ref{poissones} following \cite{FM},
\cite{CFM}.

\bigskip
Since $Q(0)=0$, the void configuration is an absorbing state for the
process $Y$. We denote by
$$
T_0=\inf\{t\ge 0: Y_t=0 \}
$$
the extinction time. In what follows,
we will assume that the process  a.s. extincts when starting from
any initial configuration:
\begin{equation}
\label{extinct}
 \forall \, \eta\in {\atomiques}\,: \,\;\;\;
\PP_\eta(T_0<\infty)=1\,.
\end{equation}
So, in our setting we assume that competition between individuals,
often due to the sharing of limited amount of resources, yields
the discrete population to extinction with probability 1.
Nevertheless, the extinction time $T_0$ can be very large compared
to the typical life time of individuals, and for some species one
can observe fluctuations of the population size for large amounts
of time before extinction (\cite{Renault}). To capture this
phenomenon, we work with the notion of quasi-stationary measure,
that is the class of probability measures that are invariant under
the conditioning to non-extinction. This notion has been
extensively studied since the pioneering work of Yaglom for the
branching process in \cite{yaglom} and the classification of
killed processes introduced by Vere-Jones in \cite{vj}. The
description of quasi stationary distributions (q.s.d. for short)
for finite state Markov chains was done in \cite{DS}. For
countable Markov chains the infinitesimal description of q.s.d. on
countable spaces was studied in \cite{np} and \cite{vandoorn}
among others, and the more general existence result in the
countable case was shown in \cite{FKMP}. For one-dimensional
diffusions there is the pioneering  work of Mandl \cite{mandl}
further developed in \cite{cms}, \cite{ms1}, \cite{Steinsaltz} and
for bounded regions one can see \cite{pinsky} among others. For
models of population dynamics and demography see \cite{catt},
\cite{Gosselin} and \cite{cattme}.

\bigskip

Let us recall the definition of a quasi-stationary distribution
(q.s.d.).

\begin{definition}
A probability measure $\nu$ supported by the set of
nonempty configurations $\atomiques^{-0}$ is said to be a q.s.d. if
\begin{equation}
\label{nan10}
\forall \; B\in {\cal B}(\atomiques^{-0}) \, : \;\;\;
\PP_\nu(Y_t\in B \, | \, T_0 >t)=\nu(B) \,,
\end{equation}
where ${\cal B}({\atomiques}^{-0})$ is the class of Borel sets of
${\atomiques}^{-0}$ and where as usual we put
$\PP_\nu=\int\limits_{{\atomiques}^{-0}} \PP_\eta d\nu(\eta)$.
\end{definition}

\medskip

When starting from a q.s.d. $\nu$, the absorption at the state $0$
is exponentially distributed (for instance see \cite{FKMP}).
Indeed, by the Markov property, the q.s.d. equality
$\PP_\nu(Y_t\in d\eta, T_0>t)=\nu(d\eta)\PP_\nu(T_0>t)$ gives
\begin{eqnarray*}
\PP_\nu(T_0\!>\!t\!+\!s)&=&\int\limits_{\atomiques^{-0}}
\PP_\nu(Y_t\in d\eta, T_0\!>\!t\!+\!s)=
\PP_\nu(T_0\!>\!t)\int\limits_{\atomiques^{-0}}\nu(d\eta)
\PP_\eta(T_0\!>\!s)\\
&=& \PP_\nu(T_0\!>\!t)\PP_\nu(T_0\!>\!s).
\end{eqnarray*}
Hence there exists $\theta(\nu)\ge 0$, the exponential decay rate
(of absorption), such that
\begin{equation}
\label{charqsd1}
\forall \, t\ge 0 \, :\;\;\;\;\PP_\nu(T_0>t)=e^{-\theta(\nu) t} \,.
\end{equation}
In nontrivial situations as ours, $0<\PP_\nu(T_0>t)<1$ (for
$t>0$), then $0<\theta(\nu) <\infty$.

\subsection{ The main results}
\label{subsub12}

Let us introduce the global quantity
\begin{eqnarray}
\label{co111}
 \tm&=&\inf_{\eta\in {\atomiques}^{-0}}\inf_{y\in
\{\eta\}} \lambda_{y}(\eta)\,.
\end{eqnarray}

\begin{theorem}
\label{mmainn}
Under the assumption
\begin{equation}
\label{assumpH}   \tn \, < \, \tm
\end{equation}
there exists a q.s.d $\nu$,  with exponential decay rate
$$
\theta(\nu)=-\log \beta \ \hbox{ with } \ \beta=\frac{\int
\EE_\eta(\|Y_1\|) \, d\nu(\eta)} {\int \| \eta\| \, d\nu(\eta)}
>0\,.
$$
\end{theorem}

This result is shown in Section \ref{prin11}. It is based on an
intermediate abstract theorem proving the existence of finite
eigenmeasures for some positive operators (Theorem
\ref{abstraitg}).

\medskip In Section \ref{sec4}, we will introduce a natural
$\sigma$-finite measure $\mu$ and show that absolute continuity
with respect to $\mu$ is preserved by the process. We study the
Lebesgue decomposition of a q.s.d. with respect to $\mu$.

\bigskip

In Section \ref{subsub56} we will study the uniform case, which is
given by
\begin{equation}
\label{uniform}
\lambda_{y}(\eta)=\lambda,\;
b_{y}(\eta)=b(1-\rho),\,\; m_{y}(\eta)=b\rho\,,
\end{equation}
where $\lambda$, $b$ and $\rho$ are positive numbers with
$\rho<1$. The property \eqref{assumpH} reads $\lambda>b$. In this
case it can be shown that $\beta=e^{-(\lambda-b)}$, so Theorem
\ref{mmainn} ensures the existence of a q.s.d. with exponential
decay rate $\lambda-b$. We will prove that this q.s.d. is  the
unique one with this decay rate, under the (recurrence) condition
\begin{equation}
\label{hyp-unif} \sigma\otimes \sigma \{(y,z)\in {\traits}^2:
g_y(z)=0\}=0\,,
\end{equation}
and that given the weights of the configuration, the locations  of
the traits under this q.s.d. are absolutely continuous with
respect to $\sigma$.

\begin{theorem}
\label{Theoch1} In the uniform case assume  that $\lambda>b$ and
\eqref{hyp-unif}. Then there is a unique q.s.d. $\nu$ on
$\atomiques^{-0}$, associated with the exponential decay rate
$\theta=\lambda-b$. Moreover $\nu$ satisfies the absolutely
continuous property,
$$
\nu\left( \vec{\eta} \in \bullet \; | \; {\overline \eta} \right)
\, << \, \sigma^{\otimes\# \eta}(\bullet) \,.
$$
\end{theorem}
In this statement,  $\vec{\eta}$ denotes the ordered sequence of
the  elements of the support $\{\eta\}$, (the compact metric space
$(\traits, d)$  being ordered in a measurable way, see Subsection
2.1), and
\begin{equation}
\label{discrete-str}
{\overline \eta}=(\eta_y: y\in \{\eta\})
\end{equation}
is the  associated sequence of strictly positive weights ordered
accordingly.

\bigskip

 In all what follows, the set $\atomiques$ will be endowed with the Prohorov metric
which makes it a Polish space (complete separable metric space).
This metric induces the weak convergence topology for which
$\atomiques$ is closed in the finite positive measure set.
 (See for example
\cite{DV-J} Chapter $7$ and Appendix).

\medskip
Let us give a general smoothness result which will be used several
times later on.
\begin{lemma}
\label{general-cont} Let $F:\atomiques \times \traits\to \RR$ be a
continuous function on $\traits$. Then the function $\hat{F}$
defined on $\atomiques \times \traits$ by
$$
\hat{F}(\eta,z) = \int\limits_{\traits} F(y,\eta,z) \eta(dy)\,,
$$
is continuous.
\end{lemma}
\begin{proof}
Let $\eta, \tilde{\eta} \in \atomiques$ and $z, z' \in \traits$.
Thus
\begin{eqnarray*}
|\hat{F}(\eta,z) - \hat{F}(\tilde{\eta},z')| &\leq& \langle  \eta,
|F(.,\eta,z)\!-\!F(.,\eta,z')|\rangle \!+\! \langle \eta,
|F(.,\eta,z')\!-\! F(.,\tilde{\eta},z')|\rangle \\
&& + |\ \langle \eta- \tilde{\eta}, F(.,\tilde{\eta},z')\rangle\,
|.
\end{eqnarray*}
Since $\traits$ is a compact set, it is immediate that the two
first  terms are  small if $z$ is close to $z'$  and $\eta$ close
to $\tilde{\eta}$. If $\tilde{\eta}$ is in a small enough
neighborhood of $\eta$, these two atomic measures have the same
weights, and the corresponding traits are close. In particular,
$\tilde{\eta}$ belongs to a compact set, and the smallness of the
last term follows by the equicontinuity of $F$ on compact sets.
\end{proof}


\section{ Poisson construction, martingale  and Feller properties}
\label{poissones}

Recall that \eqref{cont111} and \eqref{co333} are assumed. We now
give a pathwise construction of the process $Y$. As a preliminary
result, we introduce an equivalent representation of the finite
point measures as a finite sequence of ordered elements.

\subsection{Representation of the finite point measures}
\label{sec2}

Since $(\traits,d)$ is a compact metric space there exists a
countable basis of open sets $({\cal U}_i: i\in \NN=\{1,2,..\})$,
that we fix once for all. The representation
$$
{\cal R}:\traits\to \{0,1\}^\NN\,, \;\; z\to {\cal R}(z)=(c_i:i\in
\NN) \; \hbox{ with } \; c_i={\bf 1}(z\in {\cal U}_i)
$$
is an injective measurable mapping, where the set $\{0,1\}^\NN$ is
endowed with the product $\sigma-$field.
On $\{0,1\}^\NN$ we
consider the lexicographical order $\le_l$ which induces the
following order on $\traits$: $z\preceq z' \, \Leftrightarrow \,
{\cal R}(z)\le_l {\cal R}(z')$. This order relation is measurable.

\medskip

The support $\{\eta\}$ of a configuration can be ordered by
$\preceq$ and represented by the tuple
$\vec{\eta}=(y_1,...,y_{\#\eta})$ and its discrete structure
is
$\overline \eta=({\overline \eta}(k):=\eta_{y_k}: k\in \{1,...,\#\eta\})$.
Let us define  $S_0(\eta)=0$ and $S_k(\eta)=\sum\limits_{l=1}^k
\eta_{y_l}$ for $k\in \{1,...,{\#{\eta}}\}$. Remark that
$S_{\#{\eta}}(\eta)=\|\eta\|$. It is convenient to add an extra
topologically isolated point $\partial$ to $\traits$.
Now we can introduce the
functions $H^i: {\atomiques} \mapsto \traits\cup \{\partial\}$ by
$H^0(\eta)=0$ for all $\eta\in \atomiques$ and for $i\ge 1$
$$
H^i(\eta)=\begin{cases}
y_k      &\hbox{if }  i\in (S_{k-1}(\eta), S_{k}(\eta)]  \hbox{ for } k\le \#\eta\\
\partial &\hbox{otherwise}\,.
\end{cases}
$$
The functions $H^i$ are measurable. We extend the functions $b$, $\lambda$ and $m$ to
$\partial$ by putting $b_\partial(\eta)=\lambda_\partial(\eta)=m_\partial(\eta) =0$
for all $\eta\in \atomiques$.

\subsection{Pathwise Poisson construction}

Let $(\Omega, {\cal F}, \PP)$ be a  probability space in which
there are defined two independent Poisson point measures:

\begin{itemize}

\item $(ii)$ $M_1(ds,di,dz,d\theta)$ is a Poisson point measure on
$[0,\infty) \times \mathbb{N} \times \mathbb{T} \times \rit^+$,
with intensity measure $\: ds \left(\sum_{k\geq 1} \delta_k (di)
\right) d\sigma(z) d\theta\:$ (the birth Poisson measure).

\item $(i)$ $M_2(ds,di,d\theta)$ is a Poisson point measures on
$[0,\infty)\times \mathbb{N}\times\rit^+$, with the same intensity
measure $\: ds \left(\sum_{k\geq 1} \delta_k (di) \right)
d\theta\:$ (the death Poisson measure).

\end{itemize}
We denote $({\cal F}_t : t\geq 0)$ the canonical filtration
generated by these processes.

\bigskip

We define the  process $(Y_t :t\geq 0)$ as a $({\cal F}_t: t\geq
0)$-adapted stochastic process such that a.s.  and for all $t\geq
0$,
\begin{align}
Y_t &= Y_0 +
 \int\limits_{[0,t] \times \mathbb{N} \times
\mathbb{T}\times\rit^+}\ \indiq_{\{i \leq \| Y_\sm\|\}}\ \bigg\{
\delta_{H^i(Y_\sm)}\ \indiq_{\left\{\theta \leq\,
b_{H^i(Y_\sm)}g_{H^i(Y_\sm)}(z)\right\}}\notag \\ &+ \delta_{z }\
  \indiq_{\left\{
b_{H^i(Y_\sm)}g_{H^i(Y_\sm)}(z)\leq   \theta \leq\,
b_{H^i(Y_\sm)}g_{H^i(Y_\sm)}(z) + m_{H^i(Y_\sm)}
g_{H^i(Y_\sm)}(z) \right\}}\bigg\} M_1(ds,di,dz,d\theta) \notag \\
&- \int\limits_{[0,t] \times \mathbb{N} \times\rit^+}
\delta_{H^i(Y_\sm)} \ \indiq_{\{i \leq \| Y_\sm\|\}}
\indiq_{\left\{\theta \leq \,\lambda_{H^i(Y_\sm)}(Y_\sm) \right\}}
M_2(ds,di,d\theta). \label{bpe}
\end{align}
The existence of such process is proved in \cite{FM}, as well as
its uniqueness in law.  Its jump rates  are those given by
(\ref{rates11}) so this process has the same law as the process
introduced in Subsection \ref{subsub11}. In particular, the law of
$Y$ does not depend on the choice of the functions $H^i$ neither
on the order defined in Subsection \ref{sec2}.

\begin{proposition}
\label{existence} For any $\eta \in \atomiques$, any $p\geq 1$ and
$t_0>0$, there exists two positive constants $c_p$ and $b_p$
such that
\begin{eqnarray}
\label{esti-p} \EE_\eta(\sup\limits_{t\in[0,t_0]} \| Y_t\|^p )\leq
c_p \ e^{b_p t_0} <\infty.
\end{eqnarray}
\end{proposition}

\begin{proof}
Let us introduce the following hitting times,
\begin{equation}
\label{tempsatteinte}
\forall \, K\in \NN\,: \;\; T_K=\inf\{t\ge 0: \|Y_t\|\ge K\}\,.
\end{equation}

From \eqref{bpe} and neglecting the non-positive term, we easily obtain that
for any $K$,
\begin{eqnarray*}
\|Y_{t\wedge T_K}\|^p &\leq & \|Y_0\|^p +  \int\limits_{D_0}
((\|Y_\sm\|+1)^p-\|Y_\sm\|^p) M_1(ds,di,dz,d\theta)\,,
\end{eqnarray*}
where  $D_0$ is the subset of $[0,t\wedge T_K] \times \mathbb{N}
\times \mathbb{T}\times\RR_+$ which satisfies $i \leq \| Y_\sm\|$
and  $\theta \leq\,b_{H^i(Y_\sm)}g_{H^i(Y_\sm)}(z) +
m_{H^i(Y_\sm)} g_{H^i(Y_\sm)}(z)$.

\medskip

Then, by taking expectations,  from \eqref{co333} and convexity
inequality, we obtain
\begin{eqnarray*}
 \mathbb{E}_\eta(\sup_{t\leq t_0}\|Y_{t\wedge T_K}\|^p) &\leq&
 \|\eta\|^p + B^*p 2^{p-2}\int_0^{t_{0}}
(1+\mathbb{E}_\eta(\sup_{u\leq s\wedge T_K} \| Y_u\|^p))ds.
\end{eqnarray*}

Standard arguments (Gronwall's lemma ) allow us to control the growth
of the r.h.s. with respect to $t_{0}$. In particular, we have
$$
\EE_\eta(\sup\limits_{t\wedge T_K}
\| Y_t\|^p )\leq (\|\eta\|^p + p 2^{p-2} B^*) e^{p 2^{p-2}B^* t_0}.
$$
With $p=1$,  \eqref{esti-p} implies \begin{equation}
\label{decatiempo} K\;\proba_\eta\big(T_{K}<t_{0})\leq
\|\eta\|\;e^{B^{*}t_{0}}\;,
\end{equation}
and thus $T_K\to +\infty$ a.s. as $K\to +\infty$ and the process
is well defined on $\mathbb{R}_+$. Letting $K$ go to infinity
leads to the conclusion of the proof.
\end{proof}

Observe that the process $\|Y\|$ is dominated
everywhere by the integer-valued process $Z$ solution of
\begin{equation}
\label{Zeta}
Z_t=\|Y_0\| +  \int\limits_{[0,t] \times \mathbb{N} \times
\mathbb{T}\times\rit^+}\ \indiq_{\{i \leq \| Z_\sm\|\}}\
\indiq_{\left\{\theta \leq\, B^*\ g_{H^i(Z_\sm)}(z)\right\}}
M_1(ds,di,dz,d\theta),
\end{equation}
which is a birth process with rate $B^*$. This means that a.s.
$\|Y_t\|\le Z_t$. We can establish the following stronger result.

\begin{lemma}
\label{domination}
The process $\|Y\|$ is dominated
by a birth and death process with birth rate $B^*$ and death rate
$\lambda_*$. Then if we assume that $\lambda_*>B^*$, the process
$Y$ is absorbed exponentially fast.
\end{lemma}

\begin{proof}
We introduce a coupling  on the subset $\mathscr{J}$ of
$\atomiques\times\entiers$ defined by
$$
\mathscr{J}=\big\{(\eta,m)\in
\atomiques\times\entiers\,:\,\|\eta\|\le m\big\}.
$$
The coupled process is defined by its
infinitesimal generator $J$, given by the rates
$$
\begin{array}{rcl}
J(\eta,m;\eta+\delta_{y},m+1)&=& \eta_y
b_{y}(\eta),  \;\, y\in \{\eta\}\;,\\
J(\eta,m;\eta+\delta_{z},m+1)&=& G(\eta,z),
\;\,  z\not\in \{\eta\} \;,\\
J(\eta,m;\eta,m+1)&=& m\tn-
\sum\limits_{y\in\{\eta\}}\eta_{y}(b_{y}(\eta)+m_{y}(\eta))\,,\\
J(\eta,m;\eta-\delta_{y},m-1)&=&\tm \eta_y \,, \,\;
y\in \{\eta\} \,,\\
J(\eta,m;\eta-\delta_{y},m)&=& \eta_y (\lambda_{y}(\eta)-\tm)\,,
\,\; y\in \{\eta\} \,,\\
J(\eta,m;\eta,m-1)&=&\tm (m-\sum\limits_{y\in\{\eta\}}\eta_y) \,.\\
\end{array}
$$
It is immediate to check that the coordinates of this process have
respectively the law of $Y$ and the law of a birth and death
process with birth rate $\tn$ and death rate $\tm$.
On the other hand when the coupled process starts from $\mathscr{J}$
it remains in $\mathscr{J}$ forever, so the domination follows.

\medskip

In \cite{vandoorn} it is shown that the condition $\lambda_*>B^*$
implies that the birth and death chain is exponentially absorbed.
The above domination implies that so does $\|Y\|$.

\end{proof}

\medskip

It is useful to prove at this stage the following result on
hitting times that only requires the property $B^*<\infty$.

\begin{lemma}
\label{cotasup} For any $t\ge0$ and any $\eta \in
\atomiques^{-0}$, there is a number $c=c(t,\|\eta\|)\in (0,1)$
such that,
\begin{equation}
\label{ade222}
\forall \, K>0\,: \;\;\; \PP_{\eta}\big(T_K\le t\big)\le c^{-1}e^{-c\,K} \,.
\end{equation}
\end{lemma}

\begin{proof}
The proof follows immediately from the domination
of $\|Y\|$ by the birth process $Z$ introduced in (\ref{Zeta}).
Indeed, assume $\|Y_0\| <K$ and denote by $T^{Z}_M$
the smallest time
such that $Z_t\ge M$. Then $T_K\le T^{Z}_K$ a.s..
Therefore, for any $t\ge0$, and any $\eta\in \atomiques^{-0}$
$$
\PP_{\eta}\big(T_K\le t\big)\le \PP_{\|\eta\|}\big(T_K^{Z}\le
t\big)\;.
$$
For a pure birth process (see for example \cite{feller}) we have
$$
\PP_{\|\eta\|}\big(T_K^{Z}\le t\big)\le \PP_{\|\eta\|}
\big(Z_t\ge K\big)
=\sum_{m=K}^\infty \!\left(\begin{array}{c} \!m\!-\!1\!\\
\!\|\eta\|\!-\!1\!
\end{array}\right)e^{-\tn\,\|\eta\|\,t}\left(
1-e^{-\tn\,t}\right)^{m-\|\eta\|}\;.
$$
The result follows at once from this estimate.
\end{proof}

\subsection{Martingale properties}

\medskip

The process $Y$ is Markovian and we describe its infinitesimal
generator, in a weak form, using related martingales. The main
hypotheses here are the boundedness of the total birth rate per
individual (see (\ref{co333})) and the following bound for the
death individual rate: there exist $p\geq 1$ and $c>0$ such that
\begin{equation}
\label{co444}
\sup\limits_{y\in \traits} \lambda_y(\eta)\leq c \, \|\eta\|^p \,.
\end{equation}

We define the weak generator of $Y$. Given  $f:\atomiques \to \RR$,
a measurable and locally bounded function with $f(0)=0$, we define $Lf$ as
\begin{eqnarray}
\label{w_generator} Lf(\eta) &=& \sum\limits_{y\in \{ \eta\} }
\eta_y b_{y}(\eta) \left(f(\eta+\delta_{y})-f(\eta)\right)\\
\nonumber &{}&  +\sum\limits_{y\in \{ \eta \}}\eta_y m_{y}(\eta)
\int\limits_{\traits}
\left(f(\eta+\delta_{z})-f(\eta)\right)g_y(z) d\sigma(z)\\
\nonumber &{}&  +\sum\limits_{y\in \{ \eta\}}\eta_y
\lambda_{y}(\eta) (f(\eta-\delta_{y})-f(\eta))\,.
\end{eqnarray}

\begin{proposition}
\label{martingale}
Let  $f:\RR_+ \times {\atomiques} \to
\mathbb{R}$ be a measurable function such that for any
$\rho \in \atomiques$ the marginal function $f(\bullet, \rho)$
is $C^1$. We assume $f(\bullet,0)=0$ and we take $Y_0=\eta$.

\begin{itemize}
\item[(i)] If $f$ and $\partial_s f$ are  bounded on $[0,t_0]\times\atomiques$, for any
$t_0\ge 0$, then
\begin{equation}
\label{pbm2}
\mathscr{M}^f_t=:
f(t, Y_t) - f(0,\eta) - \intot (\partial_s f(s, Y_s) + Lf (Y_s)) ds
\end{equation}
is a c\`adl\`ag $({\cal F}_t: t\geq 0)$-martingale.

\item[(ii)] Moreover, if there exists a finite $p$
such that for any $t_0\ge 0$ we have
$$
\sup\limits_{0\le t\le t_0} | f(t,\eta)| +
|\partial_t f(t,\eta)| \le C(t_0) (1+\|\eta\|^p),
$$
for some finite $C(t_0)$, then $\mathscr{M}^f$ is a martingale.

\item[(iii)] If the functions
$f, \partial_s f$ are assumed to be continuous, or more
generally locally
bounded,
then $\mathscr{M}^f$ is a local martingale and for any $T_N=\inf\{t>0: \|Y_t\|\ge N\}$ the
process $(\mathscr{M}^f_{T_N\wedge t} : t\ge 0)$ is a martingale.
\end{itemize}
\end{proposition}

\begin{proof}
Let us prove the first part of the Proposition. For all $t\ge 0$
$$
f(t,Y_t)- f(0,\eta) -
\int_0^t \partial_s f(s,Y_s) ds=\sum\limits_{s\leq t} (f(s,Y_\sm +
(Y_s-Y_\sm))- f(s,Y_\sm))
$$
holds $\PP_\eta$ almost surely. A
simple computation shows that
\begin{align*}
& f(t,Y_t) - f(0,\eta) -\int_0^t \partial_s f(s,Y_s) ds=\\
& \int\limits_{[0,t] \times \mathbb{N} \times \traits\times
\mathbb{R}^+}\!\!\!\!\!\!\!\!\!\!\!\!\indiq_{\{i \leq \| Y_\sm\|\}}\
\bigg\{ \left(f(s,Y_\sm \!+ \ \delta_{H^i(Y_\sm)}) \!-\! f(s,Y_\sm)
\right) \,\indiq_{\{ \theta \leq
b_{H^i(Y_\sm)}(Y_\sm)g_{H^i(Y_\sm)}(z)\}}
\\
& \quad +  \left(f(s,Y_\sm +\ \delta_{z})\! - \! f(s,Y_\sm)\right) \,
\indiq_{\{\theta \leq m_{H^i( Y_\sm)}(Y_\sm)\,
 g_{H^i(Y_\sm)}(z)\}}\bigg\} M_1(ds,di,dz,d\theta) \\
& \, + \!\!\!\!\!\!\!\int\limits_{[0,t] \times  \mathbb{N}
\times\rit^+} \!\!\!\!\!\!\!\!\!\!\!\! \left(f(s,Y_\sm\! -\
\delta_{H^i(Y_\sm)}) \!- \!f(s,Y_\sm) \right)\! \indiq_{\{i \leq
\|Y_\sm\|, \, \theta \leq \lambda_{H^i(Y_\sm)}(Y_\sm)\}}
M_2(ds,di,d\theta),
\end{align*}
where both integrals belong to $L^1(\PP_\eta)$.
Compensating each Poisson measure,
using Fubini's Theorem, and the fact that $\int\limits_{\traits} g_y(z)
d\sigma(z)=1$, we obtain
$$
f(t, Y_t) - f(0,\eta) - \int_0^t (\partial_s f(s, Y_s) +Lf (Y_s)) ds
$$
is a martingale. The rest of the Proposition is proved by
localization arguments, justified by the result and the proof of
Proposition \ref{existence}.
\end{proof}

\subsection{Feller property of the semi-group}

Let $\N=(\N_t:t\ge 0)$ be the number of jumps for the process $Y$.
We shall prove by induction the following result.

\begin{lemma}
\label{lem0}
Assume that $f: \RR_+\times \atomiques \to \RR$ is a
bounded continuous function. Then for all $m\ge 0$
$$
(t,\eta)\to \EE_\eta(f(t,Y_t), \, \N_t=m)
$$
is a continuous function.
\end{lemma}

\begin{proof}
We notice that continuity and uniform continuity on every
$\atomiques_k$, $k\ge 1$ are equivalent because these sets are
compact. Also we have that $|f|$ is bounded on
$[0,t_0] \times \left(\bigcup_{k=1}^n \atomiques_k\right) $ for any
$t_0,n$ and we denote by $\|f\|_{t_0,n}$ its supremum on this set.
Denote by $\ell=\|\eta\|$ and $n=m+\ell+1$.

\medskip

We first prove the continuity on time. For this purpose, we assume
that $0\le u\le t\le t_0$ where we assume that $t,u$ are close
and $t_0$ is fixed. From
$$
f(t,Y_t)=f(t,Y_u) {\bf 1}_{\N_t=\N_u}+ f(t,Y_t) {\bf 1}_{\N_t\neq
\N_u}
$$
we find (recall the notation (\ref{locbo1})),
\begin{eqnarray*}
&{}& |\EE_\eta(f(t,Y_t),  \N_t=m)-\EE_\eta(f(u,Y_u),  \N_u=m)|\\
&{}& \le \sup\limits_{\|\xi\|\leq \ell+m} |f(t,\xi)-f(u,\xi)| +
\|f\|_{t_0,n} Q_+(m+\ell) \left((t-u)+o(t-u)\right)\,.
\end{eqnarray*}
Since the set $\big\{\xi\,\big|\,\|\xi\|\leq \ell+m\big\}$ is compact,
it follows from the uniform continuity of $f$ on compact sets that the
first term on the r.h.s. is small if $|t-u|$ is small.
Hence the result follows.

\bigskip

So in what follows we consider that $t=u$ and we prove continuity
on $\eta$. We will do it by induction on $m$. In the case $m=0$ we
have $\EE_\eta(f(t,Y_t),\, \N_t=0)=f(t,\eta) e^{-Q(\eta) t}$ which
is clearly continuous on $\eta$. Now we prove the induction step,
so we assume that the statement holds for $m$ and all continuous
functions $f$. We have
$$
\EE_\eta(f(t,Y_t), \N_t\!=\!m\!+\!1)\!=\! \int\limits_0^t \!
(A_1(\eta,m,t\!-\!s)\!+
A_2(\eta,m,t\!-\!s)\!+A_3(\eta,m,t\!-\!s))e^{-Q(\eta) s} \, ds,$$
where
\begin{eqnarray*}
A_1(\eta,m,t\!-\!s)&=&\sum\limits_{y\in \eta}\eta_{y} b_{y}(\eta)
\EE_{\eta +\delta_{y} }(f(t-s,Y_{t-s}),\, \N_{t-s}=m)\\
A_2(\eta,m,t\!-\!s)&=&\sum\limits_{y\in \eta}\eta_{y}\lambda_{y}(\eta)
\EE_{\eta -\delta_{y} }(f(t\!-\!s,Y_{t-s}), \N_{t-s}=m)\\
A_3(\eta,m,t\!-\!s)&=& \int\limits_{\traits} \EE_{\eta +\delta_z
}(f(t-s,Y_{t-s}), \N_{t-s}=m) G(\eta,z)\sigma(dz).
\end{eqnarray*}
Using Lemma \ref{lem0}, condition (\ref{cont111}) and Lemma
\ref{general-cont}, it is immediate that the functions $A_1$,
$A_2$ and $A_3$ are continuous in $(t,\eta)$. We conclude by the
Dominated Convergence Theorem since $f$ is bounded.
\end{proof}

\bigskip

\begin{proposition}
\label{feller}
Let $f:  \RR_+\times\atomiques \to \RR$ be a
bounded continuous function. Then
$$
(t,\eta)\to \EE_\eta(f(t,Y_t))
$$
is a continuous bounded function.
\end{proposition}

\begin{proof}
Using Lemma \ref{domination} (1) and the proof of Lemma
\ref{cotasup},  we obtain that for each $\eta\in \atomiques$,
$t>0$, there exists $a=a(t,\| \eta\|)>0$ such that
for any positive integer $K$,
\begin{eqnarray}
\nonumber
\PP_{\eta}\big(T_K^{\N}\le t\big) &=&
\PP_{\eta}\big(\N\geq M\big) \le \PP_{\eta}\big(Z\geq K +
\|\eta\|\big)\\
\label{bodim}
&=&\PP_{\|\eta\|}\big(T_K^{Z}\le t\big) \le a^{-1}e^{-a \, K} \, .
\end{eqnarray}

Assume that $\eta'$ is closed to $\eta$, and consider $u,t$ close
and smaller than $t_0$ fixed. Then
$$
\begin{array} {l}
|\EE_\eta(f(t,Y_t))-\EE_{\eta'}(f(u,Y_u))|\le
2 \|f\|\PP_{\|\eta\|}(Z_{t_0}\ge M + \|\eta\|)+\\
\quad\quad\quad\sum\limits_{m=0}^K |\EE_\eta(f(t,Y_t),\,
\N_t=m)-\EE_{\eta'}(f(u,Y_u),\, \N_u=m)|.
\end{array}
$$
The result follows by taking a large $K$, and by applying the
bound (\ref{bodim}) and Lemma \ref{lem0}.
\end{proof}

\section{ Quasi-stationary distributions}
\label{sec3}

\subsection{The process killed at $0$.}
Let us recall that the state $0$ is absorbing for the population
process $Y$. We have moreover assumed in \eqref{extinct} that the
population goes almost surely to extinction, that is
$\PP(T_0<\infty)=1$. This is in particular true if
$\lambda_*>B^*$. Our aim is the study of existence and possibly
uniqueness of a q.s.d. $\nu$, which is a probability measure on
$\atomiques^{-0}$ satisfying $\PP_{\nu}(Y_t\in B \, | \, T_0
>t)=\nu (B)$. Let us now give some preliminary results for
quasi-stationary distributions (q.s.d.).

\medskip

Since by condition \eqref{assumpH}, the process $Y$ is almost
surely but not immediately absorbed, and since starting from a
q.s.d. $\nu$, the absorption time is exponentially distributed
(see (\ref{charqsd1})), then its exponential decay rate  satisfies
$0<\theta(\nu) <\infty$. Since $0$ is absorbing it holds
$\PP_\nu(Y_t\in B)=\PP_\nu(Y_t\in B, T_0>t)$ for $B\in {\cal
B}(\atomiques^{-0})$. So, the q.s.d. equation can be written as,
\begin{equation}
\label{nant5} \forall B\in {\cal B}(\atomiques^{-0})\,,\quad \;\,
\nu(B)=e^{\theta(\nu) t} \PP_\nu(Y_t\in B)\, .
\end{equation}

From the above relations we deduce that
for all $\theta<\theta(\nu)$, $\EE_\nu(e^{\theta T_0})<\infty$.
So, for all $\theta < \theta(\nu)$, $\nu-$a.e. in ${\eta}$
it holds: $\EE_{\eta}(e^{\theta T_0})<\infty$.
Then, a necessary condition for the existence of a q.s.d. is
exponential absorption at $0$, that is
\begin{equation}
\label{nan2} \exists {\eta}\in \atomiques^{-0},\;\  \exists
\theta>0, \;\  \EE_{\eta}(e^{\theta T_0})<\infty.
\end{equation}

Let $(P_t: t\ge 0)$ be the semigroup of the process before killing
at $0$, acting on the set $C_b(\atomiques^{-0})$ of real
continuous bounded functions defined on $\atomiques^{-0}$:
$$
\forall \eta \in \atomiques^{-0}, \ \forall \, f\in
C_b(\atomiques^{-0})\,: \;\; (P_t f)(\eta)=\EE_\eta (f(Y_t),
T_0>t)\,.
$$

Let us observe that  for any continuous and bounded function
$h:\atomiques\to \RR$ and for any $\eta\in\atomiques^{-0}$, we have
\begin{equation}
\label{les2sm}
\EE_\eta (h(Y_t))=\EE_\eta (h(Y_t),
T_0>t)+h(0)\proba_\eta\big(T_0\le t\big)\;.
\end{equation}
In particular, if $h(0)=0$, we get $\ \EE_\eta (h(Y_t))=\EE_\eta
(h(Y_t), T_0>t)$.

\bigskip

We denote by $P_t^{\dagger}$  the action of the semigroup on
$\borel(\atomiques^{-0})$, defined for any positive measurable
function $f$ and any $v\in \borel(\atomiques^{-0})$ by
$$
P_t^{\dagger}v (f)=v(P_t f).
$$
From relation (\ref{nant5}) we get that a probability measure $\nu$ is a q.s.d. if and only if
there exists $\theta>0$  such that for all $t\ge 0$
$$
\nu (P_t f) = e^{-\theta t} \nu(f) \,,
$$
holds for all positive measurable function $f$, or equivalently for all $f\in
C_b(\atomiques^{-0})$. Then $\nu$ is a
q.s.d. with exponential decay rate $\theta$ if and only if it
verifies
\begin{equation}
\label{eigenmeasure} \forall t\ge 0\,: \;\;\,  P_t^{\dagger}\nu=
e^{-\theta t} \nu\,.
\end{equation}

\subsection{Some properties of q.s.d.}

Let us show that the existence of a q.s.d. will be proved if for a
fixed strictly positive time, the eigenmeasure equation
\eqref{eigenmeasure} is satisfied. In what follows we denote by
$ \petit(\atomiques^{-0})$ the set of probability measures on
$\atomiques^{-0}$.

\medskip

\begin{lemma}
\label{semco11} Let $\tnu\in \petit(\atomiques^{-0})$ and
$\beta>0$ such that $P_1^{\dagger} \tnu=\beta \tnu$. Then
$\beta<1$ and there exists $\tnu$ a q.s.d. with exponential decay
rate $\theta:=-\log \beta>0$.
\end{lemma}
\begin{proof}
From $\beta=\tnu P_1(\atomiques^{-0})=\PP_{\tnu}(T_0>1)<1$, we get
$\beta<1$, so $\theta:=-\log \beta>0$. We must show that there
exists $\nu\in \petit(\atomiques^{-0})$ such that
$P_t^{\dagger}\nu =e^{-\theta t}\,\nu$ for all $t\ge 0$. Consider,
$$
\nu=\int_{0}^{1}e^{\theta s}P_s^{\dagger} \tnu \,ds\;.
$$
For $t\in(0,1)$ we have
\begin{eqnarray*}
P_t^{\dagger} \nu&=&\int_{0}^{1} \! e^{\theta s}\, P_{t+s}^{\dagger}\tnu \,ds
=\int_{0}^{1-t}\!\! e^{\theta s} P_{t+s}^{\dagger} \tnu \,ds+
\int_{1-t}^{1} \! e^{\theta s}  P_{t+s}^{\dagger}\tnu \,ds \\
&=&\int_{t}^{1}e^{\theta (u-t)} P_{u}^{\dagger}\tnu \,du+
\int_{1}^{1+t} e^{\theta(u-t) } P_{u}^{\dagger} \tnu \,du\\
&=&e^{-\theta t}\int_{t}^{1}e^{\theta u} P_{u}^{\dagger}\tnu \,du+
e^{-\theta t}\int_{0}^{t} e^{\theta u}
e^{\theta} P_{u}^{\dagger} P_{1}^{\dagger} \tnu\,du =e^{-\theta t} \nu\;.
\end{eqnarray*}
For $t\ge 1$ we write $t=n+r$ with $0\le r<1$ and $n\in\NN$. We have
$$
P_{t}^{\dagger} \nu= P_{r}^{\dagger}P_{n}^{\dagger}\nu =
{\beta}^{n} P_{r}^{\dagger} \nu =
e^{-n\theta}  e^{-r\theta} \nu=e^{-\theta t}\nu\;.
$$
\end{proof}

Note that
\begin{equation}
\label{expI}
\theta(\nu)=\lim\limits_{t\to 0^+} \frac{1}{t}
\left(1-\PP_\nu(T_0>t)\right)=\lim\limits_{t\to 0^+}
\frac{\PP_\nu(T_0\le t)}{t}\,.
\end{equation}

In the next result we give an explicit expression for the
exponential decay rate associated to a q.s.d. We will use the
identification between $y\in \traits$ and the singleton
configuration that gives unit weight to the trait $y$.

\begin{lemma}
\label{thetaI} If $\nu\in \petit(\atomiques^{-0})$ is a q.s.d.
then its exponential decay rate $\theta(\nu)$ satisfies
\begin{equation}
\label{nan3} \theta(\nu)=\int\limits_{{\eta}\in {\atomiques}_1}
Q({\eta},0) \, \nu(d\, {\eta})= \int\limits_{\traits} Q(y,0)
d\nu(y)\,= \int\limits_{\traits} \lambda_y(y) d\nu(y)\,.
\end{equation}
\end{lemma}

\begin{proof}
Since $0$ is absorbing we get that for all fixed
$\eta\in \atomiques^{-0}$ the absorption probability
$\PP_{\eta}(T_0 \le t)$ is increasing in time $t$.
Let us denote
$$
a_2(t)=\sup\{\PP_{\eta}(T_0\le t): \, \|\eta\|= 2 \}\,.
$$
Obviously we have $a_2(s)\le a_2(t)$ when $0\le s \le t$.
We claim that
\begin{equation}
\label{bound0}
\sup\{\PP_{\eta}(T_0\le t):\, \|\eta\|\ge 2 \}\le a_2(t)
\end{equation}
Indeed let ${\widehat T}_2=\inf\{t\ge 0: \|Y_t\|=2\}$. Since the
process will be a.s. extinct for all $\eta\in \atomiques$ with
$\|\eta\|
> 2$, we have $\PP_{\eta}({\widehat T}_2 < \infty)=1$. From the
Markov property and the monotonicity in time of $a_2(t)$ we get
that for all $\eta\in \atomiques$ with $\|\eta\|> 2$,
\begin{eqnarray*}
\PP_{\eta}(T_0\le t)&=&\sum\limits_{\xi: \|\xi\|=2}
\int_0^t \PP_{\eta}({\widehat T}_2=ds, Y_{{\widehat T}_2}=\xi)
\PP_{\xi}(T_0\le t-s)\\
&\le & a_2(t)\sum\limits_{\xi: \|\xi\|=2} \int_0^t
\PP_{\eta}({\widehat T}_2=ds, Y_{{\widehat T}_2}=\xi)=a_2(t)\,.
\end{eqnarray*}
Now let us show that $a_2(t)=o(t)$, that is $\lim\limits_{t\to
0^+} {a_2(t)}/{t}=0$.

 Let $\eta\in \atomiques_2$ be a fixed
initial configuration, thus $\|\eta\|=2$.  We denote by
$A_{\downarrow}$ the subset of trajectories such that the function
$(\|Y_t\|: t\le T_0)$ is decreasing, that is at all the jumps of
the trajectory, an individual dies.  Remark that, in the
complement set $A_{\downarrow}^c$  of $A_{\downarrow}$, either at
the first or at the second jump of the trajectory,  the number of
individuals increases. Therefore, from (\ref{bound0}), the Markov
property and the monotonicity in time of $\PP_{\eta}(T_0\le t,
A_{\downarrow})$, we get that
$$
\sup\{\PP_{\eta}(T_0\le t, A_{\downarrow}^c):\, \|\eta\|= 2\}\le
\sup\{\PP_{\eta}(T_0\le t, A_{\downarrow}):\, \|\eta\|= 2\}\,.
$$

Let us now denote by $\tau$ the time of the first jump of the
process $Y$,
and $y_1$, $y_2$ are the locations of the points in $\eta$
(they can be equal). We have
$$
\PP_{\eta}(T_0\le t, A_{\downarrow})\le
\PP_{\eta}(T_0\le t, A_{\downarrow}, Y_{\tau}=y_1)+
\PP_{\eta}(T_0\le t, A_{\downarrow}, Y_{\tau}=y_2)\,,
$$
and
$$
\PP_{\eta}(T_0\le t, A_{\downarrow}, Y_{\tau}=y_i)
= \PP({\bf
e}_{\eta}+{\bf e}_{y_i}\le t, \hbox{both events are
deaths})\!=\!\int_0^t \!\!\!f_i(s)ds\,,
$$
where ${\bf e}_{\eta}$ and ${\bf e}_{y_i}$ are two independent
random variables exponentially distributed
with parameters $Q(\eta)$ and $Q(y_i)$
respectively. Moreover, conditionally to the fact that the  two
jump events occur before time $t$, the probability to obtain two
death events is ${\lambda_{y_2}(\eta)\over Q(\eta)}\times
{\lambda_{y_1}(y_1)\over Q(y_1)}$. We have for $i=1$ (a similar
computation holds for $i=2$),
$$
f_1(s)=\int_0^s \lambda_{y_2}(\eta) e^{-Q(\eta)u}
\lambda_{y_1}(y_1)e^{-Q(y_1)(s-u)}du\le Q(y_1)(1-e^{-Q(\eta)s})\,.
$$
 By using the bounds in (\ref{locbo1}) we find,
$$
\sup\{\PP_{\eta}(T_0\!\le \!t, A_{\downarrow}): \|\eta\|\!=\!2\}
\!\le \! Q_+(1)\!
\int_0^t \!\! (1\!-\!e^{-Q_+(2) s})ds\!=\!Q_+(1) \, o(t).
$$
So $a_2(t)=o(t)$ holds and from (\ref{expI}) we obtain,
$$
\theta(\nu)=\lim\limits_{t\to 0^+}\frac{1}{t}\int\limits_\traits
\PP_y(T_0\le t)d\nu(y)\,.
$$
Similar arguments as those just developed allow to get
$\PP_y(T_0\le t)=\lambda_y(y)(1-e^{-Q(y)t})+  k\, a_2(t)$, where
$k$ is a positive constant. Then the result follows.
\end{proof}

\bigskip

\section{ Proof of the existence of q.s.d.}
\label{prin11}

In this section we give a proof of Theorem \ref{mmainn}.
The proof is based upon a more general
result, Theorem \ref{abstraitg}, which shows that for a class of
positive linear operators defined in some Banach spaces, whose
elements are real functions with domain in a Polish space, there
exist finite eigenmeasures. We show Theorem \ref{mmainn} in
Subsection \ref{finpreuve}. For this purpose we construct the
appropriate Banach spaces and the operator, in order that the
eigenmeasure given by Theorem \ref{abstraitg} is a q.s.d. of the
original problem.

\subsection{An abstract result}
\label{abst11}

In this paragraph, $(\X,d)$ is a Polish metric space.
We will denote by
$C_{b}(\X)$ the set of bounded
continuous functions on $\X$. This set becomes  Banach space when
equipped with the supremum norm.

Let $S$ be a bounded positive linear operator on  $C_{b}(\X)$.
We will also make the following hypothesis.

\noindent \textsl{Hypothesis} $\hypo$:
There exists a continuous function $\varphi_{2}$ on $(\X,d)$ such that
\begin{itemize}

\item[$\hypo_1$] $\;\;\;\;\varphi_{2}\ge 1$.

\item[$\hypo_2$] $\;\;\;$ For any $u\ge0$, the set $\varphi_{2}^{-1}([0,u])$ is
compact.
\end{itemize}

It follows from $\hypo_2$ that if $(\X,d)$ is not compact,
there is a sequence $(x_{j}: j\in \NN)$
in $\X$ such that $\lim_{j\to\infty}
\varphi_{2}\big(x_{j}\big)=\infty$.

\medskip

Before stating the main result of this section we state and prove a
lemma which will be useful later on.

\begin{lemma}\label{dual}
Let $v$ be a continuous nonnegative linear form on $C_{b}(\X)$. Assume
there is a positive number $K$ such that for any function $\psi\in
C_{b}(\X)$ satisfying  $0\le \psi\le \varphi_{2}$, we have
$$
v(\psi)\le K\;.
$$
Then there exists a positive measure $\nu$ on $\X$ such that for any
function $f\in C_{b}(\X)$
$$
v(f)=\int f\;d\nu\;.
$$
\end{lemma}

\begin{proof}
Let $C_{0}(\X)$ be the set of continuous functions
vanishing at infinity.
Let $\varpi$ be a real continuous non-increasing nonnegative
function on $\real^{+}$. Assume that $\varpi=1$ on the interval
$[0,1]$ and $\varpi(2)=0$ (hence $\varpi=0$ on $[2,\infty)$). For any
integer $m$, let $v_{m}$ be the continuous positive linear form defined
on $C_{0}(\X)$ by
$$
v_{m}(f)=v\big(\varpi(\varphi_{2}/m)\;f\big)\;.
$$
This linear form has support in the set $\varphi_{2}^{-1}([0,2m])$ in
the sense that it vanishes on those functions which vanish on this
set. Note also that $\varphi_{2}^{-1}([0,2m])$ is
compact  by hypothesis $\hypo_2$. Therefore it can be
identified with a nonnegative measure $\nu_{m}$ on $\X$,
 namely for any $f\in  C_{b}(\X)$ we have
$$
v_{m}(f)=\int f\;d\nu_{m}\;.
$$
We now prove that this sequence of measures is tight. Let $u>0$ and
define the set
$$
K_{u}=\varphi_{2}^{-1}([0,u]).
$$
Again, by hypothesis $\hypo_2$, for any $u>0$ this is a compact
set. We now observe that $\ \gun_{K_{u}^{c}}\le
1-\varpi(2\,\varphi_{2}/u)\;.
$
Therefore,
$$
\nu_{m}\big(K_{u}^{c}\big)\le
  \nu_{m}\big(1-\varpi(2\,\varphi_{2}/u)\big)=
v_{m}\big(1-\varpi(2\,\varphi_{2}/u)\big) $$ $$ =
v\big(\varpi(\varphi_{2}/m)\;
\big(1-\varpi(2\,\varphi_{2}/u)\big)\big)\;,
$$
We now use the fact that the function $ \
\varphi_{2}\varpi(\varphi_{2}/m)\;
\big(1-\varpi(2\,\varphi_{2}/u)\big) $ is in $ C_{b}(\X)$ and
satisfies
$$
\frac{u}{2} \varpi(\varphi_{2}/m)\;
\big(1-\varpi(2\,\varphi_{2}/u)\big) \le \varpi(\varphi_{2}/m)\;
\big(1-\varpi(2\,\varphi_{2}/u)\big)\varphi_{2}\le \varphi_{2}
$$
to obtain from the hypothesis of the lemma that
$$
v\big(\varpi(\varphi_{2}/m)\;
\big(1-\varpi(2\,\varphi_{2}/u)\big)\big)\le
\frac{2}{u} v\big(\varpi(\varphi_{2}/m)\;
\big(1-\varpi(2\,\varphi_{2}/u)\big)\varphi_{2}\big)\le
\frac{2K}{u}\;.
$$
In other words, for any $u>0$ we have for any integer $m$
$$
\nu_{m}\big(K_{u}^{c}\big)\le \frac{2K}{u}\;.
$$
The sequence of measures $\nu_{m}$ is therefore tight, and we denote
by $\nu$ an accumulation point which is a nonnegative measure on
$\X$. We now prove that for any $f\in C_{b}(\X) $ we have
$\nu(f)=v(f)$.
For this purpose, we write
$$
v(f)=v\big(\varpi(\varphi_{2}/m)\;f\big)
+v\big((1-\varpi(\varphi_{2}/m))\;f\big)\;.
$$
We now use the inequality
$$
\varphi_{2}\ge (1-\varpi(\varphi_{2}/m))\varphi_{2}\ge
m (1-\varpi(\varphi_{2}/m))\;,
$$
to conclude using the hypothesis of the lemma (since
$(1-\varpi(\varphi_{2}/m))\varphi_{2} \in C_{b}(\X)$) that
$$
\big|v\big((1-\varpi(\varphi_{2}/m))\;f\big)\big|
\le v\big((1-\varpi(\varphi_{2}/m))\;|f|\big)
\le \|f\|\;v\big(1-\varpi(\varphi_{2}/m)\big)
\le \frac{K}{m}\;.
$$
In other words, we have for any $f \in C_{b}(\X)$
$$
\big|v(f)-\nu_{m}(f)\big|\le  \frac{K}{m}\;.
$$
From the tightness bound, we have for any $f \in C_{b}(\X)$
$$
\lim_{m\to\infty} \nu_{m}(f)=\nu(f)\;,
$$
see for example \cite{billingsley}, and therefore $\nu(f)=v(f)$
which completes the proof of the lemma.
\end{proof}

We now state the general result.

\begin{theorem}
\label{abstraitg} Assume hypotheses $\hypo_1$ and $\hypo_2$, and
assume also that there exist three  constants $c_{1}>\gamma>0$ and
$D>0$ such that
$$
S(1)\ge c_{1}
$$
and for any $\psi\in C_{b}(\X)$ with $0\le \psi\le \varphi_{2}$
$$
S\psi\le \gamma \varphi_{2}+D\;.
$$
Then  there is a probability measure $\nu$ on $\X$ such that $\nu
\circ S=\beta\nu$, with $\beta=\nu(S(1))>0$.
\end{theorem}

\begin{proof}
In the dual space $ C_{b}(\X)^{*}$, we define for any real $K>0$
the convex set $\convex_{K}$ given by
$$
\convex_{K}=\left\{v\in \, C_{b}(\X)^{*}: \,v\ge 0,\;
v(1)=1,\; \sup_{\psi\in C_{b}(\X)\,,\,
0\le \psi\le \varphi_{2}
} v(\psi)\le K \right\}\;,
$$
Note that by Lemma \ref{dual}, the elements of $\convex_{K}$ are
positive measures.

 We observe that for any $K$ large enough the
set $\convex_{K}$ is non empty. It suffices to consider a Dirac
measure $\delta_{x}$ on a point $x\in \X$ and to take $K\geq
\varphi_{2}(x)$. Since for any $K\ge 0$, $\convex_{K}$ is an
intersection of weak* closed subsets, it is closed in the weak*
topology.

\medskip

We now introduce the non-linear operator $T$ having domain
$\convex_{K}$ and defined by
$$
T(v)=\frac{v\circ S}{v\big(S(1)\big)}\;.
$$
Note that since $S(1)>c_{1}>0$, we have
$v\big(S(1)\big)\ge c_{1}v(1)$ and this operator $T$
is well defined on $\convex_{K}$. We have obviously
$T(v)(1)=1$. We now prove that $T$ maps $\convex_{K}$ into itself.
Let $\psi\in C_{b}(\X)$ with $0\le \psi\le \varphi_{2}$.
Since
$$
S\psi\le \gamma\varphi_{2}+D
$$
and obviously
$$
0\le S\psi\le \gamma\frac{\|S\psi\|}{\gamma}
$$
we get
$$
0\le S\psi\le \gamma \big(\varphi_{2}\wedge (\|S\psi\|/\gamma)\big)+D\;.
$$
Therefore  since the function $\psi'=\varphi_{2}\wedge
(\|S\psi\|/\gamma)$ satisfies  $\psi'\in C_{b}(\X)$ and  $0\le
\psi'\le \varphi_{2}$. We conclude that for $v\in \convex_{K}$
$$
T(v)\big(\psi\big)\le  \frac{\gamma
  v(\psi')+D}{c_{1}}\;.
$$
From the bound $v(\psi')\le K$ we get
$$
T(v)\big(\psi\big)\le  \frac{\gamma
  v(\psi')+D}{c_{1}}
\le\frac{\gamma}{c_{1}}\,K+\frac{D}{c_{1}}
\le K
$$
if $K>D/(c_{1}-\gamma)$.
Therefore, for any $K$ large enough, the set
$\convex_{K}$ is non empty and mapped into itself by $T$.

\medskip

It is easy to show that $T$ is continuous on $\convex_{K}$ in the weak*
topology. This follows at once from the continuity of the operator
$S$.
We can now apply Tychonov's fixed point theorem (see \cite{tycho} or
\cite{ds}) to deduce that $T$ has a fixed point. This implies that
there is a point $\nu\in\convex_{K}$ such that $\nu\circ  S=
v(S(1))\nu$. This concludes the proof of the Theorem.
\end{proof}

\subsection{Construction of the function $\varphi_{2}$ and the proof
  of Theorem \ref{mmainn}.}
\label{finpreuve}

We assume that the hypotheses of Theorem \ref{mmainn} hold. In our
application we have a semi-group $P_{t}$ acting on
$C_b(\atomiques^{-0})$. We will use  Lemma \ref{semco11} to
construct a q.s.d.   This lemma is proved using Theorem
\ref{abstraitg}  applied to $S=P_{1}$:
$$
Sf(\eta)=P_{1}f(\eta)=\EE_{\eta}\big(f(Y_{1})\,,\,T_0>1\big)\,,\;\;
\eta\in \atomiques^{-0}\,.
$$
Here the Polish metric space $(\X,d)$ of the previous paragraph
will be $(\atomiques^{-0}, d_P)$, and so the function
$\varphi_{2}$ will have domain in the set of
nonempty configurations.

We recall the elementary formula valid for any continuous and bounded
function $f$ on $\atomiques$ and any $t\ge0$
$$
\EE_{\eta}\big(f(Y_{t})\big)=
\EE_{\eta}\big(f(Y_{t})\,,\,T_0>t\big)+f(0)\proba_{\eta}
\big(T_0\le t\big)\,.
$$

We start with the following bounds.
\begin{lemma}
\label{cotas}
Let
$\overline \lambda_1=\sup\limits_{\eta\, : \,\|\eta\|=1}\sup\limits_{y\in{\eta}}
\lambda_{y}(\eta)<\infty$, then
$$
-\overline \lambda_1 \le L \indiq \le 0\;.
$$
and  for all $\ t\ge 0$,
$$
e^{-\overline \lambda_1 t} \leq P_{t} \indiq \le \indiq\;.
$$
\end{lemma}
\begin{proof}
The proof follows at once from Lemma \ref{martingale} and a computation
of $L \indiq_{\atomiques^{-0}}$ (see formula (\ref{w_generator})).
\end{proof}

\begin{lemma}
\label{cotas2}
Consider for any $a>0$ the function
$\varphi_{2}^{a}\big(\eta\big)=e^{a\|\eta\|}\indiq_{\atomiques^{-0}}(\eta)$.
Then
$$
L \varphi^{a}_{2}(\eta)\le \left( \tn\, \left(e^{a}-1\right) +
\tm\, \left(e^{-a}-1\right)\right) \,\|\eta\|\,\varphi^{a}_{2}(\eta).
$$
\end{lemma}

\medskip

\begin{proof}
We compute $L \varphi^{a}_{2}(\eta)$ using
\eqref{w_generator}. For $\eta \in \atomiques^{-0}$ we have
\begin{equation}
\label{eqa}
\begin{array}{ll}
L\varphi^{a}_{2}(\eta) &=\sum\limits_{y\in \{
\eta\}}\eta_y \big(b_{y}(\eta)+m_{y}(\eta)\big)
\left(e^{a}-1\right) e^{a\|\eta\|}\\
&+\sum\limits_{y\in \{ \eta\}}\eta_y
 \lambda_{y}(\eta)\left(e^{-a}-1\right) e^{a\|\eta\|}
- \lambda_{y}(\eta)\gun_{\|\eta\|=1}\\
&\le \tn\,\|\eta\|\,\varphi^{a}_{2}(\eta) \left(e^{a}-1\right) +
\tm\,\|\eta\|\, \varphi^{a}_{2}(\eta) \left(e^{-a}-1\right)\,.
\end{array}
\end{equation}
\end{proof}

To define the function $\varphi_2$, we will need the two following
results.

\begin{lemma}
\label{edo}
The differential equation
\begin{equation}\label{laedo}
\frac{da}{dt}=\tm\left(1-e^{-a}\right)+\tn\left(1-e^{a}\right)
\end{equation}
has two fixed points $a=0$ and $a=\log(\tm/\tn)$. The trajectory of any
initial condition $a_{0}\in (0,\log(\tm/\tn))$ is increasing in time
and converges to  $\log(\tm/\tn)$.
\end{lemma}

\begin{proof}
Left to the reader.
\end{proof}

\begin{lemma}
\label{cotas3} Assume \eqref{assumpH}. Let $a(t)$ be the solution
of (\ref{laedo}) with  initial condition $a_0\in
(0,\log(\tm/\tn))$. Then \label{uniftime} $$\sup_{t\in
\mathbb{R}_+} \mathbb{E}_{\eta}(e^{-\lambda_* t} e^{a(t)\|Y_t\|},
T_0> t) \leq e^{a_0\|\eta\|} \,.
$$
\end{lemma}

\begin{proof}

We introduce the function
$$
f(t,\eta)=e^{-\tm t}e^{a(t)\| \eta \|}\indiq_{\atomiques^{-0}}(\eta)\;,
$$
and for any integer $N$ we denote by $f^N$ the function
$$
f^N(t,\eta)=f(t,\eta)\gun_{\| \eta \|\le N}\;.
$$
Note that $f^N(t,\eta)$ is continuous with compact support
$\{\eta : \|\eta\|\leq N\}$.

Using Proposition \ref{martingale} $(iii)$ we get
$$
f^N\big(t,Y_{t\wedge T_M}\big)=f^N\big(0,Y_0\big) +\int_0^{t\wedge
T_M}\big(\partial_sf^N(s,Y_s)+
Lf^N(s,Y_s)\big)ds+\mathscr{M}^{f^N}_{t\wedge T_M}\;,
$$
where $\mathscr{M}^{f^N}$ is a martingale. Then we obtain
$$
\EE_{\eta}\left(f^N\big(t,Y_{t\wedge
T_M}\big)\right)=f^N\big(0,Y_0\big)
+\EE_{\eta}\left(\int_0^{t\wedge T_M}\big(\partial_sf^N(s,Y_s)+
Lf^N(s,Y_s)\big)ds\right)\;.
$$
Observe that if $N>M$ and $s\le T_M$ we have
$f^N(s,Y_s)=f(s,Y_s)$. Let $N$ tend to infinity to get
$$
\EE_{\eta}\left(f\big(t,Y_{t\wedge
T_M}\big)\right)=f\big(0,Y_0\big) +\EE_{\eta}\left(\int_0^{t\wedge
T_M}\big(\partial_sf(s,Y_s)+ Lf(s,Y_s)\big)ds\right)\;.
$$
Using Lemma \ref{cotas2} we have
\begin{eqnarray*}
&{}& \partial_sf(s,Y_s)+ Lf(s,Y_s)=\\
&{}& e^{-\tm s}\left(
\left(\tm\left(1\!-\!e^{-a(s)}\right)\!+\!\tn\left(1\!-\!e^{a(s)}\right)
\right)\,\|Y_s\|\!-\!\tm\right)
\varphi_2^{a(s)}(Y_s)\!+\!L\varphi_2^{a(s)}(Y_s) \le 0\,.
\end{eqnarray*}
Therefore
$$
\EE_{\eta}\left(f\big(t,Y_{t\wedge T_M}\big)\right)\le
f\big(0,Y_0\big)\;.
$$
Letting $M$ tend to infinity and by using the Monotone Convergence
Theorem we obtain,
$$
\EE_{\eta}\left(f\big(t,Y_{t}\big)\right)\le f\big(0,Y_0\big)\;.
$$
The result follows from the definition of $f$
\end{proof}

\bigskip

 We take as function $\varphi_{2}$ the function
$$\varphi_{2}=\varphi_{2}^{a(1)},$$
 for $a$  solution of
(\ref{laedo}) with  initial condition $a_0\in (0,\log(\tm/\tn))$.
The operator $S$ is given by $S=P_{1}$, and hence is positive and maps
 continuously $C_{b}(\atomiques^{-0})$ into itself.
\medskip

We must now show that $S=P_1$, and  $\varphi_{2}$ satisfy
the hypothesis of Theorem \ref{abstraitg}.

\begin{lemma} $ $

\begin{itemize}
\item[(i)] The  hypotheses $\hypo_1$ and $\hypo_2$ are satisfied.

\item[(ii)] $S(1)>c_{1}>0$, with $c_{1}=e^{-\overline\lambda_1}$.

\item[(iii)]For any $\gamma>0$, there is a constant $D=D(\gamma)>0$ such that
for any $\psi\in C_{b}(\X)$ with $0\le \psi\le \varphi_{2}$
$$
S\psi\le \gamma \varphi_{2}+ D \;.
$$
\end{itemize}
\end{lemma}

\begin{proof}
The hypotheses $\hypo_1$ and $\hypo_2$ are easy to check using the
Feller property of $P_{1}$ (see Proposition \ref{feller}).

\medskip

$(ii)$ follows at once from Lemma \ref{cotas}. We now prove $(iii)$.

Let  $\psi\in C_{b}(\X)$ with $0\le \psi\le \varphi_{2}$. We have from
Lemma \ref{cotas3}
$$
P_{1}\psi(\eta)=\mathbb{E}_{\eta}\big(\psi(Y_{1})\,,\,T_{0}>1\big)
 \le \EE_{\eta}(\varphi_{2}(Y_{1}), T_0>1)\le
e^{\lambda_*}\;e^{a_0\|\eta\|}\;.
$$
Since $a(1)>a_0 $ by Lemma \ref{edo}, for
any $\gamma>0$ there is an integer $m_{\gamma}$ such that for any
$m\ge m_{\gamma}$ we have
$$
e^{\lambda_*}\;e^{a_0\,m}\le \gamma  e^{a(1)m}\;.
$$
Therefore, for any $\eta$ we have
$$
P_{1}\psi(\eta)\le e^{\lambda_*}\;e^{a_0\|\eta\|}\le
\gamma e^{a(1)\,\|\eta\|}+
e^{\lambda_*}\;e^{a_0\,m_{\gamma}}\;.
$$
In other words, we have proved $(iii)$ with the constant
$D=e^{\lambda_*}\;e^{a_0\,m_{\gamma}}$.
\end{proof}

Theorem \ref{mmainn} follows immediately from the previous Lemma and
Theorem \ref{abstraitg}.

\section{ The process and absolute continuity }
\label{sec4}

In this section we introduce a natural $\sigma-$finite measure $\mu$.
We will show that the process $Y$ preserves the absolutely continuity
with respect to $\mu$ and that
when the process starts from any point measure after any positive time the
absolutely continuous part of the marginal distribution
does not vanish.

\subsection{The measures}

We will denote by $\widehat {\traits^{k}}$ the set of all
$k-$tuples in $\traits$ ordered by $\preceq$ defined in Subsection
\ref{sec2}. So, for $\eta\in \atomiques$, its ordered support
$\vec{\eta}=(y_1,...,y_{\#\eta})$ belongs to $\widehat
{\traits^{\#\eta}}$. The discrete structure $\overline \eta$ is an
element in $\NN^{\# \eta}$ and the set of all discrete structures
is denoted by
$$
\Sigma(\NN)=\bigcup_{n\in \ZZ_+} \NN^n\,.
$$
Here $\NN^0$ contains a unique element denoted by $0$ and it is
the discrete structure of the void configuration $\eta=0$. A
generic element of $\Sigma(\NN)$ will be denoted by $\vec q$.
Moreover for each ${\vec q}\in \Sigma(\NN)$ we put $\#{\vec q}=k$
if ${\vec q}\in \NN^k$. We put
$$
{\atomiques}_{\vec q}=\{\eta\in \atomiques: {\overline \eta}={\vec
q}\} \ \hbox{ for } {\vec q}\in \Sigma(\NN)\,,
$$
and for $B \subseteq \atomiques$,
$$
B_{\vec q}=\{\eta\in B:
{\overline \eta}={\vec q}\} \ \hbox{ for } {\vec q}\in
\Sigma(\NN)\,.
$$
In the sequel for $\vec q\in \NN^{k}$ and
$C\subseteq \widehat {\traits^k}$ we denote
\begin{equation}
\label{prod333} \{\vec q\}\times C:= \{\eta\in \atomiques:
{\overline \eta}=\vec q, \vec{\eta}\in C\}\,.
\end{equation}

We denote by $\borel_f({\atomiques})$ the set of measures on
$({\atomiques}, \B({\atomiques}))$ that give finite weight to all
sets ${\atomiques}_k$. By $\borel_f(\Sigma(\NN))$
we mean the set of measures on $\Sigma(\NN)$
giving finite weight to all the
subsets $\NN^k$, and $\borel_f(\NN))$ denotes the measures on
$\NN$ giving finite weight to all its points. Every measure
$v\in \borel_f({\atomiques})$
defines a measure ${\overline v}\in \borel_f(\Sigma(\NN))$ by
\begin{equation}
\label{igualdad1}
{\overline v}(\vec q)=v({\atomiques}_{\vec q})\,.
\end{equation}
Also $v$ defines a set of conditional measures
$v_{\vec q}\in \borel(\widehat {\traits^{\#{\vec q}}})$ by
$$
v_{\vec q}(\bullet)=\begin{cases}
0 &\hbox{if } {\overline v}(\vec q)=0,\\
v(\eta\in \atomiques: {\overline \eta}={\vec q}, \, \vec{\eta}\in
\bullet) /{\overline v}(\vec q)   &\hbox{otherwise} \,.
\end{cases}
$$
Then $v_{\vec q}\in \petit(\widehat {\traits^{\#{\vec q}}})$ is a
probability measure if ${\overline v}(\vec q)>0$.
In the case that $v\in \petit({\atomiques})$ we have
\begin{equation}
\label{chprob2}
v_{\vec q}(\bullet)=
v\left(\vec {\eta}\in \bullet \, | \, {\overline \eta}={\vec q}\right)
\,.
\end{equation}

Conversely a probability measure $v\in \petit(\atomiques)$ is
given by a probability measure ${\overline v}\in
\petit(\Sigma(\NN))$  and the family of conditional measures
$\left(v_{\overline  \eta}\in \petit(\atomiques_{\overline
\eta})\right)$ so that
$$
v(B)=\sum\limits_{{\vec q}\in \Sigma(\NN)} {\overline v}({\vec q})
\; v_{\vec q}(B_{\vec q}) ,\;\, B\in {\B}({\atomiques^{-0}})\,.
$$
In this sense
\begin{equation}
\label{difflun}
dv(\eta)={\overline v}({\overline \eta})
dv_{\overline \eta}(\eta).
\end{equation}
Let $\varphi:\atomiques\to \RR$ be a function. Observe that its
restriction to ${\atomiques}_{\vec q}$ can be identified with a
function $\varphi \big|_{{\atomiques}_{\vec q}}$ with domain in
$\widehat {\traits^{\#{\vec q}}}$ by the formula $\varphi
\big|_{{\atomiques}_{\vec q}}(\etat)=\varphi(\eta)$. Let
$\varphi:\atomiques\to \RR$ be a $v-$integrable function, we have
$$
\int_{\atomiques} \varphi(\eta)\, dv(\eta)= \sum\limits_ {\vec
q\in \Sigma(\NN)} {\overline v}({\vec q})\, \int_{\widehat
{\traits^{\#{\vec q}}}} \varphi \big|_{{\atomiques}_{\vec q}}(\vec
y) \, dv_{\vec q} (\vec y).
$$

Now we define the measure $\mu$ by,
\begin{equation}
\label{muuuu}
\mu({\entiers}^0\times \traits^0)=\mu(\{0\})=1 \hbox{ and }
\mu|_{{\entiers}^k\times \widehat {\traits^k}}=\ell^k\times
\widehat {\sigma^k}
\hbox{ for }k\ge 1 \,,
\end{equation}
where $\ell^k$ is the point measure on ${\entiers}^k$ that gives a
unit mass to every point, and $\widehat {\sigma^k}$ is the
restriction to $\widehat {\traits^k}$ of product measure
$\sigma^{\otimes k}$. Note that $v\in {\borel_f({\atomiques})}$
satisfies
$$
v\ <<\  \mu \ \Leftrightarrow \ \left( \, \forall\, {\vec q}\in
\Sigma(\NN): v_{{\vec q}}<<\widehat{\sigma^{\# {\vec q}}} \,
\right)\,.
$$
Hence, if $v\in \Si({\atomiques})$ is such that $v<< \mu$, then
$v$ is of the form
\begin{equation}
\label{lundi1} v\left(\eta\in \atomiques: {\overline \eta}={\vec
q}, \vec {\eta} \in d\,{\vec y}\right)= {\overline v}({\vec q}) \,
\varphi_{\vec q} (\vec y)\; d\widehat{\sigma^{\# {\vec q}}} ({\vec
y})\,.
\end{equation}
where ${\overline v}\in \petit(\Sigma(\NN))$ and for each fixed
${\vec q}\in \Sigma(\NN)$, $\varphi_{\vec q}({\vec \bullet})$ is a
density function in $\widehat {\traits^{\#{\vec q}}}$ with respect
to $\widehat{\sigma^{\# {\vec q}}}$.

\subsection{ Absolutely continuity is preserved }
\label{subsub41}

\begin{proposition}
\label{propo1}
The process $Y$ preserves the absolutely continuity with
respect to $\mu$, that is
\begin{equation}
 \forall \; v\in \Si(\espace), \; v<<\mu \; \Rightarrow \;\;
\forall t>0 \;\; \PP_v(Y_t\in \bullet)<< \mu (\bullet)\, .
\end{equation}
\end{proposition}

\begin{proof}
Let us define the jump time sequence,
$$
\tau_0=0 \hbox{ and } \tau_{n}=\inf\{t> \tau_{n-1}:
Y_t\neq Y_{\tau_{n-1}}\} \hbox{ for } n\ge 1.
$$
In particular $\tau=\tau_1$ is the time of the first jump. Remark
that the sequence $\tau_n$ tends a.s. to infinity, as it can be
deduced from (\ref{decatiempo}). When we need to emphasize the
dependence on the initial condition $Y_0={\eta}$ we will denote
$\tau^{\eta}_n$ and $\tau^{\eta}$ instead of $\tau_n$ and $\tau$,
respectively. We have
$$
\PP_v(Y_t\in \bullet)=\sum\limits_{n\ge 0} \PP_v(Y_t\in \bullet,
\tau_n\le t<\tau_{n+1})\,.
$$

Since the sum of absolutely continuous measures is also absolutely
continuous it suffices to prove that
\begin{equation}
\label{lun3} \PP_v(Y_t\in \bullet, \tau_n\le t< \tau_{n+1})<<\mu
\hbox{ for all }n\ge 0 \,.
\end{equation}
First, let us show the case $n=0$. For $B\in {\cal B}(\espace)$ we have
that the expression
\begin{equation}
\label{mar20} \PP_v(Y_t\in B, t< \tau)=\PP_v(Y_0\in B, t< \tau)=
\int_B v(d{\eta})\PP_{\eta}(t< \tau)
\end{equation}
vanishes if $v(B)=0$, so also when $\mu(B)=0$.

\bigskip

Before considering the case $n\ge 1$ in (\ref{lun3}) let us prove
the relation
\begin{equation}
\label{lun4}
v<<\mu \,  \Rightarrow \, \PP_v(Y_\tau \in \bullet)<<\mu.
\end{equation}
It suffices to fix $\vec q \in \Sigma(\NN)$ and to show that the
measure $\PP_v(Y_\tau \in \bullet)$ restricted to the class of
sets $(\{\vec q\} \times C : C\in {\cal B}(\widehat
{\traits^{\#{\vec q}}}))$ is absolutely continuous with respect to
$\widehat {\sigma^{\#{\vec q}}}$ (see \ref{prod333}). Let $k=\#
{\vec q}$. For $k=0$ the claim holds because $\mu(\{0\})=1$. Let
$k\ge 1$. Recall that the probability measure $v$ has the form
stated in (\ref{lundi1}), so $\varphi_{\overline \eta}=d
v_{\overline \eta} \, / d \widehat {\sigma^{\#{\eta}}}$ is the
density function on the space $\widehat {\traits^{\#{\eta}}}$.

Below, for $\eta\in \atomiques$ we denote
\begin{eqnarray}
\nonumber
&{}& \eta^{-y}=\eta-\delta_y\,,\; y\in \{\eta\}\,; \;\;\;
\eta^{+z}=\eta+\delta_z  \,, \; \forall z\in \traits\,;\\
\label{igualdad2}
&{}& {\overline \eta}{\,}^{-y}:={\overline {\eta^{-y}}}\, ; \;\;
{\overline \eta}{\,}^{+z}:={\overline {\eta^{+z}}}\,; \,\;\;
{\etat}{\,}^{-y}:={\vec{\eta^{-y}}}\, ; \;\;
{\etat}{\,}^{+z}:={\vec{\eta^{+z}}}\,.
\end{eqnarray}

We denote $\hbox{v}_k={\widehat {\sigma^{k}}}({\widehat {\traits^k}})$
(for $k=0$ we set $\hbox{v}_0=1$).
Let $k\in \NN$, ${\vec q}\in \NN^k$ and $C\in {\cal B}(\widehat
{\traits^{k}})$ be fixed. We have
\begin{eqnarray*}
&{}&\PP_v(Y_{\tau}\in \{\vec q\} \times C)= \int\limits_\atomiques
dv(\eta')
\int\limits_{\{\vec q\} \times C}{1\over Q(\eta')} Q(\eta', d\eta)\\
&=& \int\limits_C \!\!{\bf 1}({\overline \eta}\!=\!{\vec q})\!\!
\left(\sum\limits_{y\in \{\eta\} \,: \, \eta_y >1}\!\!\!\!\!
\frac{(\eta_y \!-\!1)b_{y}({\eta}^{\, -y})}{Q({\eta}^{\, -y})}
{\overline v}({\vec q}^{\, -y}) \varphi_{ {\overline \eta}^{\,
-y}}({\etat}^{\, -y})\!\right) \!\!
d\widehat {\sigma^{k}}({\etat})\\
&{}& +\int\limits_C \!\!{\bf 1}({\overline \eta}\!=\!{\vec q})\!\!
\left(\sum\limits_{y\in \{\eta\}} \!\!\! \frac{(\eta_y
\!-\!1)\lambda_{y}({\eta}^{\, +y})} { Q({\eta}^{\, +y}) }
{\overline v}({\eta}^{\, +y}) \varphi_{ {\overline \eta}^{\, +y} }
({\etat}^{\, +y})\! \right) \!\!
d \widehat {\sigma^{k}}({\etat})\\
&{}& + \frac{\hbox{v}_{k+1}}{\hbox{v}_{k}} \int\limits_C \!\!{\bf
1}({\overline \eta}\!=\!{\vec q})\!\! \left(\,\int\limits_{z\in
\traits\setminus \{\eta\}}\!\!\! \frac{\lambda_{z}({\eta}^{+z})}
{Q({\eta}^{+z})} {\overline v}( {\overline \eta}^{\,+z})
\varphi_{{\overline \eta}^{\,+z}}({\etat}^{+z})
d\sigma(z)\!\right) \!\!
d\widehat {\sigma^{k}} ({\etat})\\
&{}& +\frac{\hbox{v}_{k-1}}{\hbox{v}_{k}} \int\limits_C \!\!{\bf
1}({\overline \eta}\!=\!{\vec q})\!\! \left(\sum\limits_{y\in
\{\eta\} \,: \, \eta_y >1} \!\!\!\!\!\!\!\! {\overline v}(
{\overline \eta}^{\, -y}) \varphi_{ {\overline
\eta}^{\,-y}}({\etat}^{-y})\!\!\! \sum\limits_{y'\in
\{\eta\}\setminus \{y\}} \!\!\!\!\!\! \frac{\eta_{y'}
m_{y'}({\eta}^{-y}) g_{y'}(y)}{Q({\eta}^{-y})} \!\right) \!\!
d\widehat {\sigma^{k}}({\etat})\,,
\end{eqnarray*}
where in the last two terms we have used the following relations
$$
d\sigma(z)d\widehat{\sigma^{k}}({\etat})=
\frac{\hbox{v}_{k}}{\hbox{v}_{k+1}}
d\widehat{\sigma^{k+1}}({\etat}^{\, +z})\,, \; z\not\in \{\eta\}
$$
and
$$
d\widehat{\sigma^{k}}(\etat)= \frac{\hbox{v}_{k}}{\hbox{v}_{k-1}}
d\sigma(y') \, d\widehat{\sigma^{k-1}} ({\etat}^{\,-y})\,, \; y\in
\{ \eta\}.
$$

Hence, the relation (\ref{lun4}) is proved. An inductive argument
gives
\begin{equation}
\label{mar10} \PP_v(Y_{\tau_n}\in \bullet)<<\mu \; \hbox{ for
every } \, n \ge 1 \; .
\end{equation}

Now, let us show that (\ref{lun3}) holds for $n\ge 1$.
Denote by $F_n$ the distribution of
$\tau_n$. By the strong Markov property and Fubini theorem we get
$$
\PP_v(Y_t\in \bullet, \tau_n\le t < \tau_{n+1})=\int_0^t
\EE_v\left(\PP_{Y_{\tau_n}}(Y_{t-s}\in \bullet, \tau \ge
t-s)\right) dF_n(s)\,.
$$
which, by using relations (\ref{mar20}) and (\ref{mar10}),  is
absolutely continuous with respect to $\mu$.
\end{proof}

\subsection{ Evolution after the first mutation }
\label{subsub42}

We want to study the  absolute continuity with respect to $\mu$ of
the law  of $Y_t$ initially distributed
according to  a general measure
$v$. To this aim we will introduce the first mutation
time. Note that a mutant individual
 has a different trait from those of
its parent, so the time of first mutation is
\begin{equation}
\label{chichi}
\chi=\inf\{t\ge 0: \{Y_t\}\not\subseteq \{Y_0\} \, \}\,.
\end{equation}
When $\chi$ is finite we have $\{Y_\chi\}\neq \emptyset$,
so $(\chi<\infty)\Rightarrow (\chi < T_0)$.

\bigskip

Now, let us consider the first time where the traits of the initial
configuration disappear,
$$
\kappa=\inf\{t\ge 0: \{Y_t\}\cap \{Y_0\}=\emptyset\}\,,
$$
and for a fixed $\eta$, the first time where the traits of $\eta$
disappear, $\kappa^{\eta}=\inf\{t\ge 0: \{Y_t\}\cap
\{{\eta}\}=\emptyset\}$. When $ \{\eta\} \cap \{Y_0\} =\emptyset$,
then $\kappa^{\eta}$=0.

\medskip
 We have $\kappa\neq \chi$ except when
$\kappa=\chi=\infty$. Obviously $\kappa\le T_0$. Moreover
\begin{equation}
(\chi\!> \! \kappa) \Leftrightarrow
(\infty\!=\!\chi\!> \!\kappa)\Leftrightarrow (\chi \!> \!\kappa\!=\! T_0)
\; \hbox{ and } \;
(\kappa\!<\!T_0) \Leftrightarrow (\chi\!< \! \kappa\!< \!T_0) \,.
\end{equation}

Also note that  $(\kappa<\chi)\cap (\kappa\le t) \subseteq
(\kappa=T_0\le t)$.
\medskip
Since $\PP_{\eta}(Y_t\in \bullet, T_0\le t)=
\delta_0(\bullet)$ is concentrated at ${\eta}=0$, then
$$
\PP_{\eta}(Y_t\in \bullet, \kappa<\chi, \kappa \le t )= \delta_0(\bullet).
$$
The unique nontrivial cases are the following two ones.

\begin{proposition}
\label{ac2m}
Let ${\eta}\in \atomiques^{-0}$ and $t\ge 0$, we have:

\medskip

\noindent \item $(i)$ $\PP_{\eta}(Y_t\in \bullet , \chi < \kappa \le t<T_0)$
is absolutely continuous with respect to $\mu$ and it is concentrated
in $\atomiques^{-0}$;

\medskip

\noindent \item $(ii)$
$\PP_{\eta}(Y_t\in \bullet, t <\kappa )$ is singular with respect to $\mu$.
\end{proposition}

\begin{proof}
Let us show $(i)$. From the Markov property we have,
\begin{eqnarray}
\PP_{\eta}(Y_t\in \bullet, \chi < \kappa \le t )&=&
\sum\limits_{\xi: \emptyset \neq \{\xi\}\subseteq \{{\eta}\}}
\!\!\!
\PP_{\eta}(\chi < \kappa \! \le \! t, Y_{\chi^-}\! = \! \xi,
Y_t \! \in \! \bullet)\\
\nonumber
&=&
\sum\limits_{\xi: \emptyset \neq \{\xi\}\subseteq \{{\eta}\} }
\, \sum\limits_{y\in \{\xi\}}\!\frac{\xi_y m_y(\xi)}{Q({\xi})}
\int_0^t \!\! \PP_{\eta}
(\chi \! \in \! ds, Y_{s^-} \! = \! \xi)\! \times \\
\nonumber
&{}& \int\limits_{\traits\setminus \{\xi\}} \!\!\!\!
g_y(z) \PP_{\xi^{+z}} (Y_{t-s}\! \in \!
\bullet, \kappa^{\xi}\! \le \! t\!-\!s) d\sigma(z).
\end{eqnarray}
Hence, it is sufficient to show that for every $u>0$, ${\eta}\in
\atomiques^{-0}$ and $y\in \{{\eta}\}$, it holds
$$
\int\limits_{\traits\setminus \{{\eta}\}}
\PP_{{\eta}^{+z}}
(Y_u\in \bullet, \kappa^{{\eta}}\le u)\,g_y(z)
 \, d\sigma(z) \, <<  \mu(\bullet).
$$

By using $\int\limits_{\traits\setminus \{{\eta}\}}
\PP_{{\eta}^{+z}} \left(\{Y_u\}\cap \{{\eta}\} \neq
\emptyset,\kappa^{{\eta}} \le u\right)g_y(z)d\sigma(z)=0$, and
since the measure $\sigma$ is non-atomic, a similar proof to the
one showing Proposition \ref{propo1} works and proves the result.
Indeed, for each $t>0$, the singular part with respect to $\mu$ of
$\mathbb{P}_\eta(Y_t\in \cdot)$ is a measure on the set of atomic
measures with support contained in $\{\eta\}$ (corresponding to
death or clonal events from individuals initially alive).

\bigskip

Let us show $(ii)$. Let $\{\eta\}\subset \traits$ be the finite
set of initial traits and put $k=\#{\eta}$. Consider the Borel set
$B=\{\xi \in \atomiques^{-0}: \{\xi\}\cap \{\eta\}\neq \emptyset
\}$ and define $B_{l,n}=\{\xi \in \atomiques^{-0}: \# {\xi}=n,
|\{\xi\}\cap \{\eta\}|=l\}$ for $n\in \NN$, $l=1,..., n\land k$.
We have $B=\bigcup\limits_{n\in \NN, \, l\in \{1,..., n\land k\}}
B_{l,n}$. Since  $\sigma$ is non-atomic we have $\mu(B_{l,n})=0$
for all $l\in \{1,...n\land k\}$. On the other hand, from the
definition of $\kappa$ we have $\PP_{\eta}(Y_t\in B, t <\kappa
)=1$, and the result follows.
\end{proof}

\bigskip

Let $v\in \petit(\espace)$. We denote by
$v^t$ the distribution of $Y_t$ when
the distribution of $Y_0$ is $v$, that is
\begin{equation}
v^t(B)=\PP_v(Y_t\in B)\, , \;\,
B\in {\cal B}(\espace)\,, \; t\ge 0\,.
\end{equation}
We denote by $v=v^{\text{ac}}+v^{\text{si}}$
the Lebesgue decomposition of $v$ into its absolutely continuous part
$v^{\text{ac}}<< \mu$ and its singular part
$v^{\text{si}}$ with respect to $\mu$. For $v^t$ this decomposition
is written as $v^t=v^{t, \text{ac}}+v^{t, \text{si}}$.
As usual, $\delta_{\eta}$ is the Dirac measure at
${\eta}\in \espace$, so $\delta_{\eta}^t$ denotes the measure
$\delta_{\eta}^t (\bullet) =\PP_{\eta}(Y_t\in \bullet)$.
We will denote by $\hbox{supp}(v)$ the closed
support of a measure $v$.

\medskip

\begin{proposition}
\label{propo2}
The process $Y$ verifies:
\begin{itemize}

\item[(i)] For all $t>0$ and all ${\eta}\in \atomiques^{-0}$ we have
$\delta_{\eta}^{t, \text{ac}}(\atomiques^{-0})>0$;

\item[(ii)] For all $t>0$ and all $v\in \Si(\atomiques^{-0})$ it
holds $v^{t,\text{ac}}\ge \int_{\atomiques^{-0}}
\delta_{\eta}^{t,\text{ac}} v( d {\eta})>0$;

\item[(iii)] Assume condition (\ref{hyp-unif}). Then for all
${\eta} \in \atomiques^{-0}$ with $\{\eta\}\subseteq
\hbox{Supp}(\sigma)$ and for all $\epsilon> 0$,
the following relation holds,
$$
\forall \, t>0 \,, \;\;\;
\delta_{\eta}^{t, \text{ac}}\left(B(\eta, \epsilon)\right)>0\,,
$$
where $B(\eta, \epsilon)=\{\eta' \in \espace: \|\eta'-\eta\|<\epsilon\}$.
\end{itemize}
\end{proposition}

\begin{proof}
It suffices to show $(i)$ and $(iii)$. Let us show the first part.
Fix $t\ge 0$ and ${\eta}\in \atomiques^{-0}$. We claim that
$\PP_{\eta}(\chi < \kappa \le t  < T_0) > 0$. In fact, it suffices
to consider the event where a mutation occurs at the first jump
and after it all the initial traits disappear before $t$ and these
changes are the unique ones before $t$. This event has strictly
positive probability, so the claim is proved. Proposition
\ref{ac2m} $(i)$ gives $\PP_{\eta}(Y_t\in \bullet , \chi < \kappa
\le t<T_0) << \mu$ and we deduce,
$$
\delta_{\eta}^{t, \text{ac}}(\atomiques^{-0})\ge
\PP_{\eta}(\chi < \kappa \le t<T_0)>0\,.
$$
Then $(i)$ holds.

\bigskip

The proof of $(iii)$ is entirely similar to the proof of $(i)$
but we need some previous remarks.
 For every $y\in \traits$ we have
$\int\limits_{\traits} g_y(z)d\sigma(z)=1$, and so, $\sigma(\{z\in
\hbox{Supp}(\sigma): g_y(z)>0\})>0$. On the other hand from
condition (\ref{hyp-unif}) the set
$$
D=\{y\in \hbox{Supp}(\sigma):  \sigma(\{z\in
\hbox{Supp}(\sigma):g_y(z)>0\})=1\}
$$
verifies $\sigma(D)=1$. In particular $\sigma(z\in D: g_y(z)>0)>0$
is satisfied for all $y\in \traits$. Now, let $\{\eta\}=\{y_i:
i=1,...,k\}$ and consider the following event: a mutation occurs
at the first jump to a trait $y'\in D$, afterwords successive
mutations to the traits in $B(y_i,\epsilon)\cap D$ take place,
then for   each trait $y_i$ there are $q_i-1$ clonal births, and
finally all the initial traits and $y'$ disappear. This history
occurs  before $t$ and assume that these changes are the unique
ones that happen before $t$.

From condition (\ref{hyp-unif})
and the definition of $D$ this event has strictly positive probability
and the claim is proved.
\end{proof}

\subsection{ Decomposition of q.s.d. }
\label{subsub43}

Let us study the Lebesgue decomposition of a q.s.d. with
respect to $\mu$.

\begin{proposition}
\label{nante20}
Let $\nu$ be a q.s.d. on $\atomiques^{-0}$. Then,

\medskip

\noindent $(i)$ $\nu^{\text{ac}}\neq 0$;

\bigskip

\noindent $(ii)$ Assume Condition (\ref{hyp-unif}). Then
$\{\eta\in \atomiques^{-0}: \{\eta\}\subseteq
\hbox{Supp}(\sigma)\} \subseteq \hbox{Supp}(\nu^{\text{ac}})$;

\bigskip

\noindent $(iii)$ If $\nu^{\text{si}}\neq 0$, the probability
measure $\nu^{*\, \text{si}}:=
\nu^{\text{si}}/\nu^{\text{si}}(\atomiques^{-0})$ satisfies
\begin{equation}
\label{jeud10'}
\PP_{\nu^{*\,\text{si}}}(Y_t\in B)=e^{-\theta(\nu)t}\nu^{*\,si}(B)\,
\; \forall \; B\in {\cal B}(B^{\text{si}})\,, t\ge 0\,,
\end{equation}
where $B^{\text{si}}\in {\cal B}(\atomiques^{-0})$ is a measurable
set such that
 $\mu(B^{\text{si}})=0$ and
$\nu^{\text{si}}(B^{\text{si}})\hfill\break=
\nu^{\text{si}}(\atomiques^{-0})$.

\end{proposition}

\begin{proof}
We first note that the existence of the set $B^{\text{si}}$ is ensured by
the Radon-Nikodym decomposition theorem.
Set $H:=\atomiques^{-0}\setminus B^{\text{si}}$. Let us
show that
\begin{equation}
\label{jeud1} \forall \, t>0, \; \forall \, {\eta}\in
\atomiques^{-0}: \;\, \delta_{\eta}^{t}(H) >0.
\end{equation}
Since $\mu(B^{\text{si}})=0$ and $\delta^{t, \text{ac}}_{\eta} \, << \mu$
we have $\delta^{t, \text{ac}}_{\eta}(B^{\text{si}})=0$. Then
$\delta^{t, \text{ac}}_{\eta}(H)=
\delta^{t, \text{ac}}_{\eta}(\atomiques^{-0})$.
By Proposition \ref{propo2},
$\delta^{t, \text{ac}}_{\eta}(\atomiques^{-0})>0$
for all $t>0$ and all ${\eta}\in \atomiques^{-0}$. So
$$
\delta_{\eta}^{t}(H)\ge \delta^{t, \text{ac}}_{\eta}(H)=
\delta^{t, \text{ac}}_{\eta}(\atomiques^{-0})>0,
$$
and the assertion (\ref{jeud1}) holds.

\bigskip

Now we prove part $(i)$. We can assume $\nu^{\text{si}}\neq 0$, if
not the result is trivial. From (\ref{jeud1}) we get,
$$
\nu^{t}(H)=\int_{\atomiques^{-0}}\; \delta_{\eta}^{t}(H)
\nu(d{\eta})>0.
$$
On the other hand, from relation (\ref{nant5}) we obtain
$\nu(H)=e^{\theta(\nu)t} \nu^{t}(H)>0$. Since
$\nu^{\text{si}}(H)=0$, we necessarily have
$\nu^{\text{ac}}(H)=\nu(H)>0$, so $(i)$ holds. Now, from
Proposition \ref{propo2} $(iii)$ a similar proof as above shows
$(ii)$.

\bigskip

Let us show $(iii)$.
Let $\nu^{*\, \text{ac}}:=\nu^{\text{ac}}/\nu^{\text{ac}}(\atomiques^{-0})$.
For every $B\subseteq B^{\text{si}}$, $B\in {\cal B}(\atomiques^{-0})$, we have
\begin{equation}
\label{jeud11'}
\nu(B)=e^{\theta(\nu)t}\left(\nu^{\text{ac}}(\atomiques^{-0})
\PP_{\nu^{*\, \text{ac}}}(Y_t\in B)+
\nu^{\text{si}}(\atomiques^{-0})\PP_{\nu^{*\, \text{si}}}(Y_t\in B)\right) \,.
\end{equation}
By Proposition \ref{propo1}, $Y$ preserves $\mu$,
so  $\PP_{\nu^{*\, \text{ac}}}(Y_t\in \bullet)<< \mu$. Since
$\mu(B^{\text{si}})=0$ we get
$\PP_{\nu^{*\, \text{ac}}}(Y_t\in B^{\text{si}})=0$.
By evaluating (\ref{jeud11'}) at $t=0$ and since
$B\in {\cal B}(B^{\text{si}})$
we find $\nu^{*\, \text{si}}(B)=\nu(B)/ \nu(B^{\text{si}})$. By putting
all these elements
together we obtain relation (\ref{jeud10'}).
\end{proof}

\section{ The uniform case }
\label{subsub56}

\subsection{ The model }

In this section, we assume that the individual jump rates satisfy,
\begin{equation}
\lambda_{y}(\eta)=\lambda\,,\; b_{y}(\eta)=b(1-\rho)\,, \;
m_{y}(\eta)=b\, \rho\,, \; \;\ \forall y\in \{\eta\} \,,
\end{equation}
$\lambda$, $b$ and $\rho$ are positive numbers with $\rho<1$.
Recall that
$g:{\traits}\times {\traits} \to \RR_+$
is a jointly continuous nonnegative function
satisfying $\int\limits_{\traits}g_y(c)d\sigma(c)=1$ for all $y\in
\traits$ and the condition (\ref{hyp-unif}).

\bigskip

We observe that in this case the process of the total number of
individuals $\|Y\|=(\|Y_t\|: t\ge 0)$ is a Markov process and that
$Y_t=0\Leftrightarrow \|Y_t\|=0$, which means that the time of
absorption at $0$ of the processes $Y$ and $\|Y\|$ is the same
(note that even if the $0$'s have a different meaning, they are
identified).

\medskip

Now, in \cite{vandoorn} it is shown that there exists a q.s.d. for
the process $\|Y\|$ killed at $0$ if and only if $\lambda>b$. In
addition,  the extremal exponential decay rate of $\|Y\|$, defined
by $\sup\{\theta(\nu);\ \nu \ q.s.d.\}$,  is equal to $\lambda -b$
and there exists a unique (extremal) q.s.d. $\zeta^{\bf e}$ for
$\|Y\|$ with this exponential decay rate $\lambda-b$, given by
\begin{equation}
\label{chhc} \zeta^{\bf e}(k)=
\left(\frac{b}{\lambda}\right)^{k-1}\left(1-\frac{b}{\lambda}\right)
\, , \; \, k\ge 1.
\end{equation}

 When $\nu$ is a q.s.d. for $Y$ with exponential decay rate
$\theta(\nu)$ then the probability vector $\zeta=(\zeta(k): k\in
\NN)$ given by
\begin{equation}
\label{ecuacion0} \zeta(k)= {\nu}(\atomiques_k)\,, \; \; k\in \NN
\, ,
\end{equation}
is a q.s.d. with exponential decay rate $\theta=\theta(\nu)$,
associated with the linear birth and death process $\|Y\|$. Hence
a necessary condition for the existence of q.s.d. for the process
$Y$ is $\lambda>b$. We also deduce that all quasi-stationary
probability measures $\tilde{\nu}$ of $Y$ with exponential decay
rate $\lambda-b$ are such that
$\tilde{\nu}(\atomiques_k)=\zeta^{\bf e}(k)$, so by (\ref{chhc})
we get $$\tilde{\nu}(\varphi_1)<\infty, \ \hbox{ where }
\varphi_1(\eta)=\|\eta\|.$$

\bigskip

Now, we know from Theorem \ref{mmainn} that there exists a q.s.d.
$\nu$ with exponential decay rate
$\theta(\nu)=\nu(P_1(1))$. Moreover, it is
immediate to show that  $\varphi_1$ satisfies
$L\varphi_1=-(\lambda -b) \varphi_1$. Then, from Proposition
\ref{martingale} we get $P_1 \varphi_1=e^{-(\lambda -b)}
\varphi_1$. Hence, if $\nu$ is a q.s.d. provided by Theorem
\ref{mmainn}  its  exponential decay rate should be
$$
\theta=\lambda -b\;.
$$

\bigskip

Let us now consider the semi-group $R_{t}$ given by,
$$
R_{t}(\varphi)(\eta)= e^{\theta t} P_{t}(\varphi)(\eta) =
e^{\theta t}\EE_{\eta}\big(\varphi(Y_{t}){\bf 1}_{T_{0}>t}\big)\ ,
\;\;\, t\ge 0\,.
$$
The function $\varphi_1$ satisfies $R_{t}\varphi_1=\varphi_1$
$\nu-$a.e. for all q.s.d. $\nu$ with the exponential decay rate
$\theta$.

\bigskip

\begin{proposition}
\label{acI}
Every q.s.d. $\nu$ with exponential decay rate
$\theta=\lambda-b$ is
absolutely continuous with respect to $\mu$.
\end{proposition}

\begin{proof}
Let $\nu$ be a q.s.d. which is not absolutely
continuous. Then we
can write the Lebesgue decomposition
$$
\nu=f\mu+\xi \, \;  \hbox{ that is } \nu(B)=\int_B fd\mu + \xi(B), \;
B\in \B(E)\,,
$$
where $f$ is a nonnegative $\mu-$integrable function and
$\xi$ is a singular measure with respect to $\mu$.

\bigskip

From now on we denote by $R_{t}^{\dag}$ the dual action of $R_t$
on the set of measures defined by $(R_{t}^{\dag}v) (\varphi)=v(R_t
\varphi)$ for every measure $v\in \borel_f(\espace)$ and any positive
measurable function
$\varphi$. Since $\nu$ is a q.s.d. it is
invariant by the adjoint semi-group $R_{t}^{\dag}$, that is
$R_{t}^{\dag}\nu=\nu$, then
$$
f\mu+\xi=\nu=R_{t}^{\dag}(\nu)=R_{t}^{\dag}(f\mu)+R_{t}^{\dag}(\xi)\;.
$$
On the other hand, it follows from Proposition \ref{propo1}
that $R_{t}^{\dag}(f\mu)\ll\mu$. Therefore
$$
R_{t}^{\dag}(f\mu)\le f\mu\;.
$$

Since $\varphi_1$ is $\nu$ integrable it must
also be $f d\mu$ integrable. From the relation
$R_{t}(\varphi_1)=\varphi_1$ we get,
$$
\int \varphi_1 \,fd\mu= \int \varphi_1\, dR_{t}^{\dag}(f\mu)
\,,
$$
and since $\varphi_1$ is strictly positive, we conclude
$R_{t}^{\dag}(f\mu)=f\mu$. This implies $R_{t}^{\dag}(\xi)=\xi$.
However by Proposition \ref{propo2} (ii) (with $v=\xi$),
$R_{t}^{\dag}(\xi)$ cannot be completely singular with respect to
$\mu$ unless $\xi$ vanishes. This concludes the proof of the
proposition.
\end{proof}

\medskip

Let us now turn to the study of uniqueness.

\medskip

\begin{lemma}
\label{conjecture} Assume condition (\ref{hyp-unif}):
$\sigma\otimes\sigma (\{g=0\})=0$. Then the Borel set
$\atomiques_{(1,1)}= \{\eta\in \espace: {\vec q}=(1,1)\}$
satisfies $\mu(\atomiques_{(1,1)})>0$. Moreover, for any q.s.d.
$\nu$ with exponential decay rate $\theta=\lambda-b$ and for any
Borel set $B$ with $\mu(\atomiques_{(1,1)}\cap B)>0$, we have
$\nu(\atomiques_{(1,1)} \cap B)>0$.
\end{lemma}

\begin{proof}
Let us consider a q.s.d. $\nu$ with exponential decay rate
$\theta=\lambda-b$. Lemma \ref{thetaI} implies that the
restriction $\nu_{(1)}$ of $\nu$ to $\atomiques_1=\{\eta\in
\espace: \|\eta\|=1\}$, does not vanish. On the other hand by the
previous result, it is absolutely continuous with respect to
$\sigma$. Then,
$$
d\nu_{(1)}=f_{1} d\sigma
$$
for some nonnegative function $f_{1}$ that does not
vanish on a set of $\sigma$ positive measure.

\bigskip

For any function  $f$  in $C_{b}(\atomiques)$ such that
$f(0)=0$ and such that is supported in a compact set, that is
$f(\eta)=0$ for all
$\|\eta\|$ large enough, it follows from
$\nu(P_{t}f)=\exp(-\theta t)\,\nu(f)$ that
$$
\nu(L f \indiq_{\atomiques^{-0}})=-\theta\; \nu(f)\;.
$$
Since this is true for any such function, we get (with $\overline
\eta=(\eta_y: y \in \{\eta\})$ as defined in (\ref{discrete-str}),
and notation (\ref{igualdad1}) and (\ref{igualdad2})),
\begin{eqnarray}
\nonumber
-\theta d\nu_{{\overline \eta}}(\etat) &=&
b(1\!-\!\rho)\sum\limits_{y: \eta_y>1}
\big(\eta_y-1){\overline \nu}({\overline \eta}^{\, -y})
d\nu_{{\overline \eta}^{\, -y}}(\etat)\\
\nonumber
&{}&
+\lambda\sum\limits_{y\in \{\eta\}}\!
\big(\eta_y\!+\!1\big){\overline \nu}({\overline \eta}^{\, +y})
d\nu_{{\overline \eta}^{\, +y}}(\etat)
\!+\!\lambda
{\overline \nu}({\overline \eta}^{\, +z})
\!\!\!\!\!\!
\int\limits_{z\in \traits\setminus \{{\vec y}\}} \!\!\!\!\!\!
d\nu_{{\overline \eta}^{\, +z}} (\etat^{\,  +z})\\
\nonumber
&{}&
+b\,\rho \sum\limits_{y: \eta_y=1}
{\overline \nu}({\overline \eta}^{\, -y})
\sum\limits_{y'\in \{\eta\}\setminus \{y\}}
\eta_{y'}\, g_{y'}(y) \; d\nu_{{\overline \eta}^{\, -y}}
(\etat^{\, -y})\, d\sigma(y)\\
\label{dualu}
&{}&
-(\lambda+b)\left(\sum\limits_{y\in \{\eta\}} \eta_y \right)\;
{\overline \nu}({\overline \eta})\, d\nu_{{\overline \eta}}(\etat) \;.
\end{eqnarray}

It follows from equation (\ref{dualu}) applied to the measure $\nu$ and
solving for $\nu_{\vec q}$ that for some constant $C>0$ we have
$$
d\nu_{(1,1)}\big(y_{1},y_{2}\big)\ge C
\left(f_{1}(y_{1})g_{y_{1}}(y_{2})+f_{1}(y_{2})g_{y_{2}}(y_{1})\right)\,
d\sigma(y_{1})\,d\sigma(y_{2})\;.
$$
Using this lower bound in the equation for $\nu_{(1)}$, we get for
some constant $C'>0$
$$
f_{1}(y)\ge C'\int f_{1}(u)\, g_u(y)\,d\sigma(u)\;.
$$
Therefore for some constant $C''>0$ we have the estimate
\begin{eqnarray*}
&{}& d\nu_{(1,1)}\big(y_{1},y_{2}\big)\ge
C''\,d\sigma(y_{1})\,d\sigma(y_{2})\times\\
&{}& \quad\quad
\left(g_{y_{1}}(y_{2})\! \int \!\!f_{1}(u) g_u(y_{1})\,d\sigma(u)+
g_{y_{2}}(y_{1})\!\int \!\! f_{1}(u) g_u(y_{2})d\sigma(u)\right)\,.
\end{eqnarray*}
Let $B$ be a Borel set with $\mu(\atomiques_{(1,1)} \cap B)>0$. By the
identification
between $\atomiques_{(1,1)}$ and ${(1,1)}\times \widehat {\traits^{2}}$
it can be assumed that $B\subseteq \widehat {\traits^{2}}$. We get from
Fubini's
Theorem
$$
\nu(\atomiques_{(1,1)} \cap B)\ge C''\int\limits_{\widehat
{\traits^{3}}}{\bf 1}_{B}(y_{1},y_{2}) g_{y_{1}}(y_{2})\,
f_{1}(u)\,
g_u(y_{1})\,d\sigma(u)\,d\sigma(y_{1})\,d\sigma(y_{2})\;.
$$
Therefore, if $\nu(\atomiques_{(1,1)} \cap B)=0$, we must have
$$
{\bf 1}_{B}(y_{1},y_{2}) \,
g_{y_{1}}(y_{2})\, f_{1}(u)\,
g_u(y_{1})=0\; \;\,  \sigma\times\sigma\times\sigma-\hbox{a.e.}\,,
$$
which implies from the hypothesis (\ref{hyp-unif}) on $g$
$$
{\bf 1}_{B}(y_{1},y_{2}) \, f_{1}(u)=0 \; \;\,
\sigma\times\sigma\times\sigma-\hbox{a.e.}\,.
$$
However this implies $f_{1}=0\;$ $\sigma-$a.e.,
which is a contradiction.
\end{proof}

\bigskip

\begin{proposition}
\label{uniIV} There is a unique q.s.d. associated with the
exponential decay rate $\theta=\lambda-b$.
\end{proposition}

\begin{proof}
Let $\nu$ and $\nu'$ be two different q.s.d. with the exponential
decay rate $\theta$. We can write the Lebesgue decomposition
$$
\nu'=f\nu+\xi
$$
with $f$ a nonnegative measurable function and $\xi$ a singular
measure with respect to $\nu$. Assume $\xi\neq0$.
Applying $R_{t}^{\dag}$ we
get
$$
\nu'=f\nu+\xi=R_{t}^{\dag}(f\nu)+R_{t}^{\dag}(\xi)\;.
$$
If $f$ is bounded, since $\nu$ is a q.s.d., we have
$R_{t}^{\dag}(f\nu)\ll \nu$. In the general case, the same result
holds by approximating $f$ by an increasing sequence of
nonnegative functions. Therefore, we must have
$$
R_{t}^{\dag}(f\nu)\le f\nu\;.
$$
Integrating the function $\varphi_1$ as before, we conclude that
$R_{t}^{\dag}(f\nu)=f\nu$, and therefore  $R_{t}^{\dag}(\xi)=\xi$.

\bigskip

Then, we have two q.s.d. $\nu$ and $\xi$ with exponential decay
rate $\theta=\lambda-b$, which are mutually singular. We claim
that this is excluded by Lemma \ref{conjecture}. Indeed let $B$ be
a measurable subset such that $\xi(B)=\nu(B^c)=0$. Then $\nu(B\cap
\atomiques_{(11)})= \nu(\atomiques_{(11)})>0$. Since $\nu << \mu$
we get  $\mu(B)\ge \mu(B\cap \atomiques_{(11)})>0$.
 From Proposition \ref{conjecture}
we deduce $\xi(B\cap \atomiques_{(11)})>0$ which is a
contradiction. Namely $\xi=0$ and we conclude that $\nu'=f\nu$.
Let us now show that $f\equiv 1$, which will yield $\nu'=\nu$ and
so will conclude the uniqueness result.

\bigskip

Recall the following notation on the ordered lattice of measures
$\borel_f(\atomiques^{-0})$: $\; |v|=v^+ + (-v)^+$ with
$v^+=\max(v,0)$. Since the linear operator $R_t^{\dag}$ is
positive it holds $|R_t^{\dag}(v)|\le R_t^{\dag}|v|$, that is for
all positive and measurable functions $\varphi$, it holds $\int \varphi \,d
|R_t^{\dag}(v)|\le \int \varphi  \,d R_t^{\dag}|v|$. Moreover,
when there exists a couple of sets $A_1, \, A_2$ such that
$R_t^{\dag}(v)(A_1)> 0 > R_t^{\dag}(v)(A_2)$ and
$R_t^{\dag}|v|(A_1)>0$, $R_t^{\dag}|v|(A_2)>0$, this inequality
becomes strict and we put $|R_t^{\dag}(v)|< R_t^{\dag}|v|$. This
means that
$$
\forall \; \varphi>0\,,\; \int \varphi  \,d R_t^{\dag}|v|<\infty\,
\hbox{ where } \;\;\;\; \int \varphi \,d |R_t^{\dag}(v)|<\int
\varphi \,d R_t^{\dag}|v| \,.
$$

\medskip

Assume that $\nu(f\neq 1)>0$, which implies $\mu(f\neq 1)>0$.
Thus, the sets $A_1=\{f<1\}$ and $A_2=\{f> 1\}$ fulfill the
requirements for the signed measure $v=\nu-\nu'$. Then, by using
that $\nu$ and $\nu'$ are $R_t^{\dag}$ invariant, we find
$$
\big|\nu-\nu'\big|=\big|R_{t}^{\dag}\big(\nu\big)
-R_{t}^{\dag}\big(\nu'\big)\big| <
R_{t}^{\dag}\big(\big|\nu-\nu'\big|\big)\,.
$$
Since $\varphi_1$ is $R_t$ invariant and $\int \varphi_1
\,d\big|\nu-\nu'\big|<\infty$, we find
$$
\int \varphi_1\,d\big|\nu-\nu'\big|<
\int \varphi_1 \,dR_{t}^{\dag}\big(\big|\nu-\nu'\big|\big)
=\int \varphi_1 \,d\big|\nu-\nu'\big|\;,
$$
which is a contradiction since $\varphi_1\ge 1$. The result is
shown.
\end{proof}

The proof of Theorem \ref{Theoch1} is now complete,
it follows from Propositions \ref{acI} and \ref{uniIV}.

\section*{Acknowledgments}

The authors acknowledge the partial support given by the
Millennium Nucleus Information and Randomness P04-069-F,
CMM FONDAP and CMM BASAL projects. S. Mart\'{\i}nez thanks
Guggenheim Fellowship and the hospitality of Ecole Polytechnique,
Palaiseau, and Sylvie M\'el\'eard the ECOS-CONICYT project.

\bigskip

\noindent PIERRE COLLET

\noindent {\it CNRS Physique Th\'eorique, Ecole Polytechnique, F- 91128
Palaiseau Cedex, France.} e-mail: collet@cpht.polytechnique.fr

\medskip

\noindent SERVET MART\'INEZ

\noindent {\it Departamento Ingenier{\'\i}a Matem\'atica and Centro
Modelamiento Matem\'atico, Universidad de Chile,
UMI 2807 CNRS, Casilla 170-3, Correo 3, Santiago, Chile.}
e-mail: smartine@dim.uchile.cl

\medskip

\noindent SYLVIE M\'EL\'EARD

\noindent  {\it Ecole Polytechnique, CMAP, CNRS-UMR7641,  F- 91128
Palaiseau Cedex, France. } e-mail: meleard@cmap.polytechnique.fr

\medskip

\noindent JAIME SAN MART\'IN

\noindent {\it Departamento Ingenier{\'\i}a Matem\'atica and Centro
Modelamiento Matem\'atico, Universidad de Chile, UMI 2807 CNRS,
Casilla 170-3, Correo 3, Santiago, Chile.}
e-mail: jsanmart@dim.uchile.cl

\end{document}